\documentclass[british]{article}
\usepackage{ae,aecompl}
\usepackage[T1]{fontenc}
\usepackage[latin9]{inputenc}
\usepackage{geometry}
\usepackage{amsmath}
\usepackage{amsthm}
\usepackage{amssymb}

\makeatletter
\theoremstyle{plain}
\newtheorem{thm}{\protect\theoremname}[section]
\theoremstyle{definition}
\newtheorem{defn}[thm]{\protect\definitionname}
\theoremstyle{remark}
\newtheorem{rem}[thm]{\protect\remarkname}
\theoremstyle{plain}
\newtheorem{lem}[thm]{\protect\lemmaname}
\theoremstyle{plain}
\newtheorem{prop}[thm]{\protect\propositionname}
\theoremstyle{plain}
\newtheorem{cor}[thm]{\protect\corollaryname}

\@ifundefined{date}{}{\date{}}

\usepackage{bbm}

\usepackage{babel}
\providecommand{\definitionname}{Definition}
\providecommand{\lemmaname}{Lemma}
\providecommand{\propositionname}{Proposition}
\providecommand{\remarkname}{Remark}
\providecommand{\theoremname}{Theorem}

\makeatother

\usepackage{babel}
\providecommand{\corollaryname}{Corollary}
\providecommand{\definitionname}{Definition}
\providecommand{\lemmaname}{Lemma}
\providecommand{\propositionname}{Proposition}
\providecommand{\remarkname}{Remark}
\providecommand{\theoremname}{Theorem}

\begin{document}
\global\long\def\IN{\mathbb{N}}%
\global\long\def\II{\mathbbm{1}}%
\global\long\def\IZ{\mathbb{Z}}%
\global\long\def\IQ{\mathbb{Q}}%
\global\long\def\IR{\mathbb{R}}%
\global\long\def\IC{\mathbb{C}}%
\global\long\def\IP{\mathbb{P}}%
\global\long\def\IE{\mathbb{E}}%
\global\long\def\IV{\mathbb{V}}%

\title{Detection of an Arbitrary Number of Communities\\
in a Block Spin Ising Model}
\author{Miguel Ballesteros\thanks{IIMAS-UNAM, Mexico City, Mexico}, Rams\'es
H. Mena\footnotemark[1], Jos\'e Luis P\'erez\thanks{CIMAT, Guanajuato, Mexico},
and Gabor Toth\footnotemark[1] \footnote{gabor.toth@iimas.unam.mx}}
\maketitle
\begin{abstract}
We study the problem of community detection in a general version of
the block spin Ising model featuring $M$ groups, a model inspired
by the Curie-Weiss model of ferromagnetism in statistical mechanics.
We solve the general problem of identifying any number of groups with
any possible coupling constants. Up to now, the problem was only solved
for the specific situation with two groups of identical size and identical
interactions, see \cite{BRS2019,LoweSchu2020}. Our results can be
applied to the most realistic situations, in which there are many
groups of different sizes and different interactions. In addition,
we give an explicit algorithm that permits the reconstruction of the
structure of the model from a sample of observations based on the
comparison of empirical correlations of the spin variables, thus unveiling
easy applications of the model to real-world voting data and communities
in biology.
\end{abstract}


\section{Introduction and Results}

In the influential article \cite{BRS2019}, Berthet, Rigollet, and
Srivastava introduced the block spin Ising model consisting of two
groups previously defined in the statistical mechanics literature
in \cite{CGh2007} to the field of community detection, and the authors
showed that for two groups of identical size and identical coupling
constants within each group the exact recovery of the community structure
was possible with high probability using a sample of observations
assumed to be generated by the model. The number of observations required
depends on the regime (or phase) of the model. Of the three regimes
of the model, the high and low temperature regimes where analysed
in \cite{BRS2019}. Subsequently, L\"owe and Schubert supplied the
missing critical regime in \cite{LoweSchu2020}, and demonstrated
that exact recovery was possible there, too. Thus, the problem of
community detection for more than two groups and/or different coupling
constants remained open. We solve the problem for a wide variety of
situations as far as the number of groups in the population, the sizes
of the groups, and interactions between voters within each group and
between different groups are concerned, and show that reconstructing
the group structure is possible given a certain number of observations.
The procedure we present is completely constructive and easy to implement.
Thus, our results in the present article allow for the application
of community detection algorithms to real-world data. Possible applications
are the identification of social classes, cultural groups, or sympathisers
of political parties, or other characteristics which may not be readily
available in data. Using the community detection algorithm will potentially
allow social scientists to better understand societal structures derived
from common preferences or interdependencies. Biologists can employ
the algorithms for the problem of distinguishing species based on
the observation of characteristics of different specimens. An important
point that should be noted is that the techniques presented in this
article can in principle also be applied to any other voting model
with several communities, for which the pair correlations between
votes can be deduced.

The literature on the detection of communities in stochastic models
is extensive. Among the models analysed, graphical models (i.e. Markov
random fields) are most prominent. See \cite{Lauritze1996} for a
description of how Markov random fields encode dependencies between
random variables. Among graphical models, the most studied are the
stochastic blockmodel and the block spin Ising model. The stochastic
blockmodel (introduced in \cite{HolBlaLe1983}) has been analysed
thoroughly. See \cite{Abbe2018} for an up to date exposition and
\cite{MosNeeSl2015,MosNeeSl2016,GaMaZhZh2017,AminLevi2018} for further
important results about this model. In the stochastic blockmodel,
we observe a single realisation of the random graph and assume that
interactions between individuals are independent. This is very different
from the block spin Ising model which is the subject of the present
article.

Beyond the mathematical interest inherent to these models and the
problem of community detection, there are also important applications
in the field of image recognition \cite{GemaGema1986,Besag1986},
biology and genetics \cite{LaurShee2003,ChenYuan2006,BallGarr2022},
recommendation systems \cite{LinSmiYo2003}, and natural language
processing \cite{MannSchu1999}, among others. These applications
are especially important given the ubiquitous use of large amounts
of data. Other applications include the use of these models applied
to problems in social sciences, such as sociological behaviour \cite{GoZhFiAi2010},
migration \cite{CGh2007}, economic decision making \cite{BD2001,OEA2018,LSV2020},
and political science \cite{KT2021c}. Community detection algorithms
applied to these problems allow researchers to understand the socioeconomic
and political structure of societies and dependencies between different
groups and classes, as well as the structures underlying species diversity
and interdependence.

Compared to the stochastic blockmodel, there has been considerably
less work about the community detection problem in Block spin Ising
models, also referred to as multi-species mean-field models \cite{FedeCont2011}
or multi-group Curie-Weiss models \cite{KirsToth2020,KirsToth2020b,KirsToth2022b}.
In these models, there is a heterogeneous population of $N\in\IN$
individuals subdivided into $M\in\IN$ groups and each individual
casts a binary vote in an election. We can also interpret the binary
choice as a biological characteristic of the individual which can
be observed and measured. The votes or measurements are random variables
which are allowed to depend on each other (with different degrees
of dependencies within each group and across group boundaries). These
dependencies are regulated by the coupling constants between votes.
We will formally present the model in Section \ref{subsec:model}.
An observer has access to the votes of all individuals in a number
of elections, referenda, or measurements. The population is assumed
static, and the observer sees the patterns in the voting behaviour.
These patterns manifest in the way certain subsets of voters frequently
tend to vote alike or contrary to each other. The observer's problem
consists of reconstructing the group structure of the population from
the observed votes or measurements alone.

Rather than taking a single realisation of a random graph as in the
stochastic blockmodel, in the block spin Ising model we observe a
sample of $n$ realisations of the model, i.e. $n$ voting configurations
of the shape $\left(x_{1},\ldots,x_{N}\right)\in\left\{ -1,1\right\} ^{N}$
which are realisations of the random vector $\left(X_{1},\ldots,X_{N}\right)$
assumed to be independent and distributed according to the probability
measure of the block spin Ising model defined in Definition \ref{def:CWM}.
Each realisation of $\left(X_{1},\ldots,X_{N}\right)$ contains the
votes or measurements of the entire population in a single election,
referendum, or observation. The problem we study consists of reconstructing
the structure of the model in terms of assigning each of the random
variables $X_{i}$, $i\in\IN_{N}$, that represents the vote of one
individual to one of the $M$ groups. Having access to a certain number
of observations permits us to recover the structure with high probability
for the possible structures of practical importance the model can
assume.

\subsection{\label{subsec:model}The Block Spin Ising Model}

The block spin Ising model is defined for $N\in\IN$ $\left\{ -1,1\right\} $-valued
random variables $X_{i}$ indexed by $i\in\IN_{N}$, commonly referred
to as spins. Since the application to voting and the observation of
biological characteristics of a heterogeneous population is the main
motivation of the study in this article, we will instead speak of
votes and voting configurations. The voters are sorted into $M\in\IN$
groups. Each group $l\in\IN_{M}$ is of size $N_{l}\in\IN$, such
that $\sum_{l=1}^{M}N_{l}=N$. The group identification function 
\begin{align}
\iota:\IN_{N}\to\IN_{M}\label{eq:iota_fn}
\end{align}
assigns each voter $i\in\IN_{N}$ to their respective group $\iota(i)\in\IN_{M}$.
Hence, the definition of the group sizes above implies $\left|\iota^{-1}\left(\left\{ l\right\} \right)\right|=N_{l}$,
where for any countable set $A$ the cardinality of $A$ is denoted
by $\left|A\right|$.
\begin{defn}
\label{def:group_sizes}We define the group size parameters for each
group $l$ 
\[
\alpha_{l}:=\frac{N_{l}}{N}.
\]
\end{defn}

\begin{rem}
\label{rem:alpha}We will assume these constants $\alpha_{l}$, $l\in\IN_{M}$,
exist and are positive. This implies in particular, that we can assume
each group consists of at least two different individuals for large
enough $N$, a fact we will use in the proofs of our results.
\end{rem}

For any voting configuration $\left(x_{1},\ldots,x_{N}\right)\in\left\{ -1,1\right\} ^{N}$,
the Hamiltonian $\mathbb{H}:\left\{ -1,1\right\} ^{N}\rightarrow\IR$
is given by 
\begin{equation}
\mathbb{H}\left(x_{1},\ldots,x_{N}\right):=-\frac{1}{2}\sum_{l,m=1}^{M}\frac{J_{l,m}}{\sqrt{N_{l}N_{m}}}\sum_{i\in\iota^{-1}\left(\left\{ l\right\} \right)}\sum_{j\in\iota^{-1}\left(\left\{ m\right\} \right)}x_{i}x_{j}.\label{eq:Hamiltonian}
\end{equation}
This Hamiltonian allows for different interactions between pairs of
votes, such that if one belongs to group $l$ and the other to group
$m$, they interact by a coupling constant $J_{l,m}\in\IR$.\footnote{These coupling constants subsume the inverse temperature parameter
$\beta$ found in the single-group Curie-Weiss model (see e.g. \cite[ Chapters IV and V]{Ell1985}
for a thorough discussion of the classical Curie-Weiss model). In
fact, in the special case of $M=1$, the definition of $\mathbb{H}$
reduces to the Hamiltonian of the Curie-Weiss model.} We note that, depending on the signs of the coupling parameters $J_{l,m}$,
the value of $\mathbb{H}$ varies with the voting configuration. In
the present context of voting, we interpret $\mathbb{H}\left(x_{1},\ldots,x_{N}\right)$
as a measure of conflict in society surrounding a particular issue
which gives rise to votes $\left(x_{1},\ldots,x_{N}\right)\in\left\{ -1,1\right\} ^{N}$.
If all coupling parameters are positive, there are two configurations
that have the lowest possible level of conflict: the unanimous configurations
$\ensuremath{(-1,\ldots,-1)}$ and $(1,\ldots,1)$. All other configurations
receive higher values $\mathbb{H}\left(x_{1},\ldots,x_{N}\right)$.
The highest levels are achieved when the votes are evenly split in
each group (or closest to it in case of odd group sizes). We define
the coupling matrix $J$ to have entries equal to the constants above:
\begin{equation}
J:=\left(J_{l,m}\right)_{l,m\in\IN_{M}}\in\IR^{M\times M}.\label{eq:J}
\end{equation}

\begin{defn}
Let $A>0$ stand for the statement that $A\in\IR^{M\times M}$ is
a symmetric positive definite matrix, i.e. the Euclidean inner product
$\left\langle Ax,x\right\rangle $ is positive for all $x\in\IR^{M},x\neq0$,
and let $A\geq0$ mean $A$ is positive semi-definite, i.e. $\left\langle Ax,x\right\rangle \geq0$
holds for all $x\in\IR^{M}$.
\end{defn}

\begin{rem}
We will assume $J>0$ throughout this article. This assumption can
be interpreted as the existence of stronger coupling between voters
belonging to the same group versus voter coupling between groups.
This assumption also allows for negative coupling between different
groups, which models an antagonistic relationship between groups.
\end{rem}

\begin{defn}
\label{def:CWM}Let $J\in\IR^{M\times M}$ with $J>0$ and $\mathbb{H}$
as defined in \eqref{eq:Hamiltonian}. The block spin Ising model's
probability measure $\mathbb{P}$, which gives the probability of
each of the $2^{N}$ voting configurations, is defined by 
\begin{align}
\mathbb{P}\left(X_{1}=x_{1},\ldots,X_{N}=x_{N}\right) & :=Z^{-1}e^{-\mathbb{H}\left(x_{1},\ldots,x_{N}\right)}\label{eq:CWM}
\end{align}
for all $\left(x_{1},\ldots,x_{N}\right)\in\left\{ -1,1\right\} ^{N}$,
where $Z$ is a normalisation constant which depends on $N$ and $J$.
\end{defn}

\subsection{Results}

The model has three distinct regimes, in each of which the model behaves
in a distinct fashion. This is reflected in the limiting distribution
of the vector of suitably normalised sums of votes $\left(S_{1},\ldots,S_{M}\right)$,
where $S_{l}$ is the sum of the votes in group $l\in\IN_{M}$. This
limiting distribution is distinct in each of the three regimes. These
results can be found in several articles, e.g. \cite{FedeCont2011,KnLoScSi2020,KirsToth2022b}.
The three regimes are called (in adaptation of the corresponding terms
for the classical single-group Curie-Weiss model) the high temperature,
the critical, and the low temperature regime. The high temperature
regime is characterised by the difference between the identity matrix
$I\in\IR^{M\times M}$ and the coupling matrix $J$ being a positive
definite matrix, i.e. $I-J>0$. The model is in the critical regime
when $I-J\geq0$ but $I-J\ngtr0$. Finally, the low temperature regime
is equivalent to $I-J\ngeq0$. High temperature corresponds to high
disorder, meaning the voters tend to have a mind of their own with
weak dependence of the votes. Low temperature corresponds to strong
couplings between votes.

We will study the problem of detecting the $M$ communities in the
model in the high temperature and the low temperature regimes. In
the latter case, we will make an assumption (see Definition \ref{def:non-crit})
about the Hessian matrices at the minima of the function $F:\IR^{M}\rightarrow\IR$
defined by
\begin{equation}
F(x)=\frac{1}{2}x^{T}J^{-1}x-\sum_{l=1}^{M}\alpha_{l}\ln\cosh\left(\frac{x_{l}}{\sqrt{\alpha_{l}}}\right),\quad x\in\IR^{M},\label{eq:F}
\end{equation}
which plays a crucial role in the analysis of the block spin Ising
model. It appears in the de Finetti representation of the probability
measure $\IP$ (see \cite[Theorem 32]{KirsToth2022b} for more details).
The minima of $F$ are of particular importance, and their location
depends on the regime the model is in.
\begin{defn}
We will use the symbol $C$ for positive constants which are independent
of $N$ but may depend on the coupling matrix $J$ and the group size
parameters $\alpha_{l}$. We make no claim as to the precise value
of these constants, and in fact they may change from one line of a
calculation to the next.
\end{defn}

\begin{thm}
\label{thm:high_temp}In the high temperature regime, i.e. for $I-J>0$,
set $H:=J^{-1}-I$. Then there is a positive constant $\boldsymbol{C}_{\textup{high}}$
such that for all $N\in\IN$
\[
\left|\mathbb{E}\left(X_{i}X_{j}\right)-\frac{H_{\iota(i),\iota(j)}^{-1}}{N}\right|\leq\boldsymbol{C}_{\textup{high}}\frac{\left(\ln N\right)^{(6+M)/2}}{N^{2}}\quad i,j\in\IN_{N},i\neq j,
\]
holds, where we use the notation $H_{\iota(i),\iota(j)}^{-1}=\left(H^{-1}\right)_{\iota(i),\iota(j)}.$
\end{thm}

This theorem is proved in Section \ref{sec:Proof-thm_bounds}.

As in the high temperature regime, in the low temperature regime,
i.e. $I-J\ngeq0$, we address the non-critical case. Our definition
of the low temperature non-critical case is similar to the corresponding
definition for high temperature.
\begin{defn}[Low temperature non-critical case]
\label{def:non-crit}In the case that $I-J\ngeq0$, we say that the
model is in the low temperature regime. Moreover, if the Hessian of
$F$ at every point where the minimum is attained is positive definite,
we say that the model is non-critical.
\end{defn}

\begin{lem}
\label{lem:minima}In the low temperature non-critical case, the number
of points where the minimum of $F$ is attained is finite. We denote
them by 
\begin{align}
\left\{ z^{(k)}\right\} _{k\in\IN_{K}},\label{eq:minima_F}
\end{align}
and the corresponding Hessian of $F$ at the point $z^{(k)}$ is denoted
by $H_{k}$ and is invertible for every $k\in\IN_{K}$.
\end{lem}

\begin{proof}
Since $F$ is dominated by the term $\frac{1}{2}x^{t}J^{-1}x$ for
large $\left\Vert x\right\Vert $, it follows that all points where
the minimum is attained must be contained in a compact set $B\subset\IR^{M}$
(we recall that $J^{-1}$ is positive definite). Suppose that $z$
is a point where the minimum is attained. The Hessian of $F$ is positive
definite at $z$ by assumption. This implies that there is a neighbourhood
$U\subset\IR^{M}$ of $z$ where $F$ attains its minimum only at
$z$. We obtain that the set of points where the minimum is reached
consists of isolated points, and it is a closed set because $F$ is
continuous. This set cannot be infinite, otherwise it would contain
an accumulation point of it since $B$ is compact, and this accumulation
point would not be isolated.
\end{proof}
We define
\begin{align}
\vec{Z}_{l} & :=\frac{1}{\sqrt{\sum_{k}\frac{1}{\det(H_{k})^{1/2}}}}\begin{pmatrix}\frac{1}{\det(H_{1})^{1/4}}\tanh\left(z_{l}^{(1)}\right)\\
\frac{1}{\det(H_{2})^{1/4}}\tanh\left(z_{l}^{(2)}\right)\\
\vdots\\
\frac{1}{\det(H_{K})^{1/4}}\tanh\left(z_{l}^{(K)}\right)
\end{pmatrix}\in\IR^{K},\quad l\in\IN_{M}.\label{eq:Zl}
\end{align}

\begin{rem}
\label{rem:Zl}Notice that the map
\begin{align*}
y\in\IR^{K} & \mapsto\frac{1}{\sqrt{\sum_{k}\frac{1}{\det(H_{k})^{1/2}}}}\begin{pmatrix}\frac{1}{\det(H_{1})^{1/4}}\tanh\left(y_{1}\right)\\
\frac{1}{\det(H_{2})^{1/4}}\tanh\left(y_{2}\right)\\
\vdots\\
\frac{1}{\det(H_{K})^{1/4}}\tanh\left(y_{K}\right)
\end{pmatrix}\in\IR^{K}
\end{align*}
is bijective. We assume that there are no $l$ and $m$ with $l\neq m$
such that $z_{l}^{(k)}=z_{m}^{(k)}$ holds for every $k=1,\ldots,K$.
This is equivalent to assuming that 
\begin{align*}
\vec{Z}_{l} & \ne\vec{Z}_{m},\quad l\neq m.
\end{align*}
The above implies that for every $l\neq m$ either $\left\Vert \vec{Z}_{l}\right\Vert \ne\left\Vert \vec{Z}_{m}\right\Vert $
or $\left\Vert \vec{Z}_{l}\right\Vert =\left\Vert \vec{Z}_{m}\right\Vert $
and $\left\langle \vec{Z}_{l},\vec{Z}_{m}\right\rangle <\left\Vert \vec{Z}_{l}\right\Vert ^{2}=\left\Vert \vec{Z}_{m}\right\Vert ^{2}$,
where the inequality derives from the non-collinearity of $\vec{Z}_{l}$
and $\vec{Z}_{m}$ and the equality condition in the Cauchy-Schwarz
inequality.
\end{rem}

\begin{thm}
\label{thm:low_temp}Assume the model is in the low temperature regime
and non-critical as per Definition \ref{def:non-crit}, i.e. the Hessian
$H_{k}$ of $F$ at $z_{k}$ is positive definite for all $k$. Then
the function $F$ defined in \eqref{eq:F} has a finite number of
minima, $z^{(1)},\ldots,z^{(K)}\in\IR^{M}$. There is a positive constant
$\boldsymbol{C}_{\textup{low}}$ such that for all $N\in\IN$
\[
\left|\mathbb{E}\left(X_{i}X_{j}\right)-\left\langle \vec{Z}_{\iota(i)},\vec{Z}_{\iota(j)}\right\rangle \right|\leq\boldsymbol{C}_{\textup{low}}\frac{\left(\ln N\right)^{\left(M+3\right)/2}}{\sqrt{N}},\quad i,j\in\IN_{N},i\neq j.
\]
\end{thm}

This theorem is proved in Section \ref{sec:Proof-thm_bounds}.

The main results of this article, the proof that the community detection
problem has a solution with high probability and the algorithms for
the reconstruction of the group structure of the model are found in
Sections \ref{subsec:High-Temp} and \ref{subsec:Low-Temp}, respectively.
Theorems \ref{thm:high_temp} and \ref{thm:low_temp} allow us to
calculate approximations for the correlations $\IE\left(X_{i}X_{j}\right)$
between votes for large $N$. We will use these approximations as
a benchmark for the empirical correlations obtained from a sample
of observations $x^{(1)},x^{(2)},\ldots,x^{(n)}\in\left\{ -1,1\right\} ^{N}$
of voting configurations from the model:
\[
\frac{1}{n}\sum_{t=1}^{n}x_{i}^{(t)}x_{j}^{(t)},\quad i\neq j
\]
(cf. formulae \eqref{stat} and \eqref{eq:emp_corr}). From the asymptotic
analysis of the block spin Ising model in Theorems \ref{thm:high_temp}
and \ref{thm:low_temp}, we know the large $N$ value of $\IE\left(X_{i},X_{j}\right)$,
and by the law of large numbers the empirical correlations converge
to these values as $n$ goes to infinity. The empirical correlations
also satisfy a large deviations principle (see Lemmas \ref{High}
and \ref{Low} and Proposition \ref{PropSanov}) which allows us to
upper bound the probability of a significant deviation of the empirical
correlations from these values by a function which exponentially decays
to 0 with the number of observations $n$. Thus, with high probability,
we obtain a sample of observations that are typical in that there
are no large deviations of the empirical correlations from their expected
values. Then we use an iterative algorithm to define an equivalence
relation on the set of voters $\IN_{N}$ that corresponds to the underlying
group structure represented by $\iota$ which is assumed to be unknown
to the observer of the voting configurations $x^{(1)},x^{(2)},\ldots,x^{(n)}$.
We next describe the algorithm informally for the high temperature
regime (the low temperature regime algorithm is structurally similar).
See the proofs of Theorems \ref{HighIdent} and \ref{LowIdent} for
a rigorous description and Section \ref{subsec:Example} for an example
of the application of this algorithm.

We take a correlation 
\[
\IE\left(X_{i^{*}}X_{j^{*}}\right)=\max_{i\neq j}\left\{ \IE\left(X_{i}X_{j}\right)\right\} 
\]
which according to Corollary \ref{cor:CS_H}, in the high temperature
regime, corresponds to voters $i^{*},j^{*}\in\IN_{N},i^{*}\neq j^{*}$,
who both belong to the same group, namely group $l\in\IN_{M}$ with
the largest value $H_{l,l}^{-1}$ (cf. Theorem \ref{thm:high_temp}).
We then identify the largest empirical correlations, which according
to Lemma \ref{Low} belong to those voters $i$ and $j$ who indeed
belong to said group, i.e. $\iota(i)=\iota(j)=l$. Having identified
the voters belonging to group $l$, we remove from our set of empirical
correlations all those elements that correspond to voter pairs which
include at least one of the indices just identified as belonging to
group $l$. Then we pick the group $l'\in\IN_{M}$ with the next largest
value $H_{l',l'}^{-1}$, and repeat the last step. In each step, we
identify at least one group of voters. Hence, the algorithm terminates
after at most $M$ steps. Thus, we provide an explicit algorithm of
how to reconstruct the structure of the model.

The main theorems of this article are
\begin{thm}
[Group Identification and Reconstruction Procedure in the High Temperature Regime]\label{HighIdent}Let
the model be in the high temperature regime. There is a fixed natural
number $N_{{\rm high}}$ and a constant $\boldsymbol{\delta}>0$ (cf.
Definition \ref{DefinitionNd}) such that if $N\geq N_{{\rm high}}$
and $n\in\IN$, a sample of $n$ observations $x^{(1)},x^{(2)},\ldots,x^{(n)}$
allows us to recover the group partition with high probability. The
probability that we cannot recover the group partition is bounded
above by the exponentially decaying function 
\begin{align*}
e_{\textup{high}}(n) & :=N^{2}(n+1)^{2}e^{-\frac{1}{8}\left(\frac{1}{8N}\boldsymbol{\delta}\right)^{2}n}.
\end{align*}
\end{thm}

and
\begin{thm}
[Group Identification and Reconstruction Procedure in the Low Temperature Regime]\label{LowIdent}Let
the model be in the low temperature regime and non-critical. Let $z^{(1)},\ldots,z^{(K)}\in\IR^{M}$
be the minima of the function $F$, and assume there are no $l$ and
$m$ with $l\neq m$ such that $z_{l}^{(k)}=z_{m}^{(k)}$ holds for
every $k=1,\ldots,K$. There is a fixed natural number $N_{{\rm low}}$
and a constant $\boldsymbol{\gamma}>0$ (cf. Definition \ref{DefinitionNdl})
such that if $N\geq N_{{\rm low}}$ and $n\in\IN$, a sample of $n$
observations $x^{(1)},x^{(2)},\cdots,x^{(n)}$ allows us to recover
the group partition with high probability. The probability that we
cannot recover it is bounded above by the following exponentially
decaying function 
\begin{align*}
e_{{\rm low}}(n) & :=N^{2}(n+1)^{2}e^{-\frac{1}{8}\left(\frac{1}{8}\boldsymbol{\gamma}\right)^{2}n}.
\end{align*}
\end{thm}

\begin{rem}
The upper bounds for the probabilities that the community detection
problem has a solution are tight in the sense that Sanov's Theorem
provides a corresponding lower bound as well, and said lower bound
features an exponential term
\[
e^{-\frac{1}{8}\left(\frac{1}{8N}\boldsymbol{\delta}'\right)^{2}n}\quad\textup{and}\quad e^{-\frac{1}{8}\left(\frac{1}{8}\boldsymbol{\gamma}'\right)^{2}n},
\]
 respectively, for some $\boldsymbol{\delta}'\geq\boldsymbol{\delta}$
in the high temperature regime and some $\boldsymbol{\gamma}'\geq\boldsymbol{\gamma}$
in the low temperature regime.
\end{rem}

\begin{rem}
In the original articles \cite{BRS2019,LoweSchu2020} concerning the
community detection problem in the block spin Ising model, the two
groups considered where identical in all respects -- group sizes
and coupling constants. Therefore, community detection was only possible
in terms of finding an equivalence relation on the set of voters $\IN_{N}$
such as those constructed in the proofs of Theorems \ref{HighIdent}
and \ref{LowIdent}. For the more general situation of several groups
of arbitrary sizes and coupling constants, we can go beyond this solution
and identify the equivalence classes detected and specific groups
of the model. One such case is when all groups are of different sizes,
i.e. $N_{l}\neq N_{m}$ for all $l\neq m$. Another example is $H_{l,l}^{-1}\neq H_{m,m}^{-1}$
for all $l\neq m$. In fact, there is a way to explicitly identify
the groups in most cases. Also, the algorithm allows identifying the
group structure in case there is imperfect information about the structure
of the model in terms of group sizes and coupling constants.
\end{rem}

In the next subsection, we present an example for the application
of the community detection algorithm to the problem of extremist political
movements. The rest of this article consists of Section \ref{sec:Exact-Recovery}
where we formally state and prove the community detection results,
Section \ref{sec:Proof-thm_bounds} in which we prove Theorems \ref{thm:high_temp}
and \ref{thm:low_temp}, and finally the Appendix which includes some
results we use in the proofs in order to make the article as self-contained
as possible.

\subsection{\label{subsec:Example}Detection of the Members of Extremist Political
Parties}

In this section, we describe how our community detection algorithms
can be employed to detect members of political movements or parties.
In many countries, extremist political parties are forbidden by the
authorities. Membership is usually prohibited and heavily penalised.
As such, admitting to membership risks prosecution up to lengthy prison
sentences. Even if an extremist political party is not prohibited,
or else it is not a formal party but more of a loose movement or association
of like-minded people, it may still be the case that members do not
openly identify as such. Censure from mainstream segments of the population
leads to the phenomenon of self-censoring and avoidance of stating
their membership openly. Thus, if we have the task of quantifying
the membership numbers of extremist parties in a certain population
of $N$ people, asking for their political affiliation directly may
not lead to accurate numbers, as extremist affiliations will likely
be underrepresented among the responses. Instead, we take the indirect
route of applying a questionnaire  with $n$ questions with binary
responses. These can be yes/no questions but also questions asking
about two possible solutions to current political or societal problems.

Suppose there are two political parties in a country. Party 1 is an
extremist political movement shunned by sympathisers of party 2, which
is a mainstream moderate party, and `party 3', the remainder of
the population which is not politically engaged. We endeavour to assign
each respondent correctly to one of the three parties. Suppose we
have a sample $x^{(1)},\ldots,x^{(n)}$ of $n$ observations of voting
configurations
\[
x^{(t)}=\left(x_{1}^{(t)},\ldots,x_{N}^{(t)}\right)\in\left\{ -1,1\right\} ^{N},\quad t\in\IN_{n},
\]
obtained from the responses to our questionnaire for the entire population
of $N\in\IN$ voters. The assumption about the population structure
implies there are three groups in our block spin Ising model (cf.
Definition \ref{def:CWM}). The coupling matrix $J=\left(J_{l,m}\right)_{l,m\in\IN_{3}}$
defined in \eqref{eq:J} is assumed to be given by
\[
J=\left(\begin{array}{ccc}
0.7 & 0.2 & 0.2\\
0.2 & 0.6 & 0.1\\
0.2 & 0.1 & 0.5
\end{array}\right),
\]
which is positive definite. Moreover, it is easily checked that the
matrix
\[
I-J=\left(\begin{array}{ccc}
0.3 & -0.2 & -0.2\\
-0.2 & 0.4 & -0.1\\
-0.2 & -0.1 & 0.5
\end{array}\right)
\]
is positive definite. Therefore, the model exhibits weak interactions
between voters by virtue of being in the high temperature regime.
We set $H:=J^{-1}-I$ (cf. Theorem \ref{thm:high_temp}), calculate
\[
H^{-1}=\left(\begin{array}{ccc}
13.6154 & 9.23077 & 7.69231\\
9.23077 & 7.46154 & 5.38462\\
7.69231 & 5.38462 & 5.15385
\end{array}\right),
\]
and order the diagonal entries of $H^{-1}$:
\[
H_{1,1}^{-1}>H_{2,2}^{-1}>H_{3,3}^{-1}.
\]
We have
\begin{align*}
\eta & :=\min_{l\ne m}\left\{ \max\left\{ H_{l.l}^{-1},H_{m,m}^{-1}\right\} -H_{l,m}^{-1}\right\} =2.07692,\\
\xi & :=\min_{l\ne m}\left\{ \left.\left|H_{l.l}^{-1}-H_{m,m}^{-1}\right|\;\right|\;H_{l.l}^{-1}\neq H_{m,m}^{-1}\right\} =2.30769,
\end{align*}
and set 
\begin{align*}
\boldsymbol{\delta} & :=\min\left\{ \eta,\xi\right\} =2.07692.
\end{align*}
Next, we calculate the empirical correlations between two votes for
each pair $i,j\in\IN_{N}$, $i\neq j$,
\[
\bar{y}_{i,j}:=\frac{1}{n}\sum_{t=1}^{n}x_{i}^{(t)}x_{j}^{(t)}.
\]
Inspecting the empirical correlations, we notice that there is a cluster
of $\bar{y}_{i,j}$ around the value $\frac{H_{1,1}^{-1}}{N}$, contained
in the open interval $\left(\frac{H_{1,1}^{-1}}{N}-\frac{\boldsymbol{\delta}}{2},\frac{H_{1,1}^{-1}}{N}+\frac{\boldsymbol{\delta}}{2}\right)$.
We find that this cluster is composed of $\frac{N_{1}\left(N_{1}-1\right)}{2}$
empirical correlations of the total $\frac{N\left(N-1\right)}{2}$
correlations corresponding to the population as a whole. Since $H_{1,1}^{-1}$
is the largest diagonal entry, the $N_{1}$ indices $i$ to be found
in the cluster of $\bar{y}_{i,j}$ around the value $\frac{H_{1,1}^{-1}}{N}$
correspond to the members of group 1, i.e. members of the extremist
political party. We now exclude all indices $i$ belonging to extremists
from our set of empirical correlations and thus obtain a reduced set
of $\frac{\left(N_{2}+N_{3}\right)\left(N_{2}+N_{3}-1\right)}{2}$
of $\bar{y}_{i,j}$. We inspect this set of correlations and find
that there is a cluster of $\bar{y}_{i,j}$ around the value $\frac{H_{2,2}^{-1}}{N}$,
in an open interval $\left(\frac{H_{2,2}^{-1}}{N}-\frac{\boldsymbol{\delta}}{2},\frac{H_{2,2}^{-1}}{N}+\frac{\boldsymbol{\delta}}{2}\right)$,
which contains $\frac{N_{2}\left(N_{2}-1\right)}{2}$ correlations.
The $N_{2}$ indices to be found in these correlations belong to members
of the mainstream political party. We exclude all $N_{2}$ of these
indices from our correlations. The remaining $\frac{N_{3}\left(N_{3}-1\right)}{2}$
correlations featuring $N_{3}$ indices belong to non-political members
of the population. We note that the fortuitous clustering of the empirical
correlations around the values $\frac{H_{l,l}^{-1}}{N}$ is, in fact,
a high probability event, provided the sample size $n$ is large enough.
(Cf. Theorems \ref{HighIdent} and \ref{LowIdent}.) This concludes
the reconstruction of this model's group structure from the voting
data in the sample.

\section{\label{sec:Exact-Recovery}Exact Recovery of the Group Structure}

In this section we prove the two main Theorems \ref{HighIdent} and
\ref{LowIdent} of this article. We will use the Theorems \ref{thm:high_temp}
and \ref{thm:low_temp}, which are proved in Section \ref{sec:Proof-thm_bounds}.

\subsection{Large Deviations Theory}

Here we will define the probability space on which the samples of
any possible size $n\in\IN$ generated by the model defined in Definition
\ref{def:CWM} live.

Let 
\begin{align*}
\Xi & :=\{-1,1\}
\end{align*}
be equipped with the $\sigma$-algebra $\mathcal{F}$ of all subsets
of $\Xi$. Let
\[
\Omega_{1}:=\prod_{i=1}^{N}\Xi=\Xi^{\IN_{N}}.
\]
 Moreover, we define 
\begin{align*}
\Omega_{n} & :=\prod_{i=1}^{n}\Omega_{1}=\Omega_{1}^{\IN_{n}}\quad n\in\IN,
\end{align*}
provided with the $\sigma$-algebra $\mathcal{F}_{n}$ of all subsets.
We set 
\begin{align*}
\Omega & :=\prod_{i=1}^{\infty}\Omega_{1}=\Omega_{1}^{\mathbb{N}},
\end{align*}
with the $\sigma$-algebra $\boldsymbol{\mathcal{F}}$ generated by
the cylinder sets of the from $\mathcal{U}\times\Omega_{1}^{\{n+1,n+2,\ldots\}}$,
$\mathcal{U}\in\Omega_{n}$ and $n\in\mathbb{N}$.

We denote by 
\begin{align*}
\omega & =\left\{ \omega^{(t)}\right\} _{t\in\mathbb{N}}
\end{align*}
the elements of $\Omega$. Here, $\omega^{(t)}\in\Omega_{1}$ for
every $t$. We set the projections $X^{(t)}:\Omega\to\Omega_{1}$
given by 
\begin{align*}
X^{(t)}(\omega) & =\omega^{(t)}\in\Omega_{1},\quad t\in\IN,
\end{align*}
and
\[
X_{i}^{(t)}(\omega)=\omega_{i}^{(t)}\in\Xi,\quad t\in\IN,i\in\IN_{N}.
\]
The random vector $X^{(t)}$ corresponds to observation $t\in\IN_{n}$.
We define a product measure $\boldsymbol{\nu}$ on $\left(\Omega,\boldsymbol{\mathcal{F}}\right)$
using Kolmogorov's extension theorem such that each $X^{(t)}$ has
a Curie-Weiss distribution according to Definition \ref{def:CWM}:
for all $x\in\Omega_{1}$ and all $t\in\IN$, 
\[
\boldsymbol{\nu}\left(\left\{ \omega\in\Omega\,|\,\omega^{(t)}=x\right\} \right)=Z^{-1}e^{-\mathbb{H}\left(x_{1},\ldots,x_{N}\right)}.
\]

We define the random variables 
\begin{align*}
Y_{i,j}^{(t)}:=X_{i}^{(t)}X_{j}^{(t)} & ,\quad i,j\in\IN_{N},i\neq j.
\end{align*}
We will next define a probability measure $\boldsymbol{\mu}_{i,j}$
on $\Omega$ under which, for fixed $i$ and $j$, $Y_{i,j}^{(1)},Y_{i,j}^{(2)},\ldots$
are independent and identically distributed (i.i.d.) random variables,
for fixed $t\in\IN$, $Y_{i,j}^{(t)},Y_{i',j'}^{(t)},\ldots$ are
in general neither independent nor identically distributed. We identify
$X_{i}^{(1)}\equiv X_{i}$ with the random variables whose joint distribution
is given by $\IP$ in Definition \ref{def:CWM}. First we define the
probability measure $\mu_{i,j}:\mathcal{F}\to[0,1]$ by 
\begin{align*}
\mu_{i,j}(\{r\}) & :=\boldsymbol{\nu}\left(\left\{ \omega\in\Omega\,\left|\,\omega_{i}^{(1)}\omega_{j}^{(1)}=r\right.\right\} \right).
\end{align*}
For every $n\in\mathbb{N}$, we define the product measure on $\mathcal{F}_{n}$
by 
\begin{align*}
\mu_{i,j}^{n} & :=\mu_{i,j}^{\otimes(n)}.
\end{align*}
Using Kolmogorov's extension theorem, we define the measure $\boldsymbol{\mu}_{i,j}$
on $\boldsymbol{\mathcal{F}}$ such that 
\[
\boldsymbol{\mu}_{i,j}\left(\mathcal{U}\times\Omega_{1}^{\{n+1,n+2,\ldots\}}\right)=\mu_{i,j}^{n}(\mathcal{U})
\]
for all such cylinder sets $\mathcal{U}\times\Omega_{1}^{\{n+1,n+2,\ldots\}}$
described above.

We denote by $L_{n}^{Y_{i,j}}:\mathcal{F}\to[0,1]$ the empirical
measure associated to the process $\left(Y_{i,j}^{(t)}\right)_{t\in\mathbb{N}}$,
i.e. 
\begin{align}
L_{n}^{Y_{i,j}}(\{r\})=\frac{1}{n}\left|\Big\{ s\in\IN_{n}\;\Big|\;Y_{i,j}^{(s)}=r\Big\}\right|, & \quad r\in\Xi.\label{p1}
\end{align}
We denote by $\mathcal{M}_{1}\left(\Xi\right)$ the set of probability
measures on $\Xi$, equipped with the total variation topology. Let
$C\left(\Xi\right)$ be the set of real-valued functions on $\Xi$
and we identify it with $\mathbb{R}^{2}$ using the map $f\mapsto(f(1),f(2))$.
We provide $C\left(\Xi\right)$ with the topology of $\mathbb{R}^{2}.$
We use the symbol $I\in C\left(\Xi\right)$ for the identity $I(r)=r,r\in\Xi,$
and for every $f\in C\left(\Xi\right)$ and every $\nu\in\mathcal{M}_{1}\left(\Xi\right)$
we define 
\[
\nu(f):=\int f\,\textup{d}\nu.
\]
Therefore, we identify every measure $\nu$ with a continuous real-valued
function on $C\left(\Xi\right).$ We set 
\begin{align*}
\bar{Y}_{i,j}^{(n)} & :=\frac{1}{n}\sum_{t=1}^{n}Y_{i,j}^{(t)}=L_{n}^{Y_{i,j}}(I),
\end{align*}
the average of $Y_{i,j}^{(t)}$ over all $t\in\IN_{N}$. We set for
every $\varepsilon\in(0,1)$
\begin{align}
\mathcal{M}_{1}\left(\Xi\right)\supset\mathcal{O}_{i,j}(\varepsilon) & :=\Big\{\nu\in\mathcal{M}_{1}\left(\Xi\right)\;\left|\;\left|\nu(I)-\mu_{i,j}(I)\right|\geq\varepsilon\right.\Big\},\label{calo}
\end{align}
which is a closed set that does not contain $\mu_{i,j}$. Note that
$\mu_{i,j}(I)=\mathbb{E}\left(X_{\iota(i)}X_{\iota(j)}\right)$. We
set 
\begin{align*}
\Omega\supset\mathcal{O}_{i,j}^{n}(\varepsilon) & :=\left\{ \omega\in\Omega\;\left|\;L_{n}^{Y_{i,j}(\omega)}\in\mathcal{O}_{i,j}(\varepsilon)\right.\right\} .
\end{align*}

For every two measures, $\nu,\theta\in\mathcal{M}_{1}\left(\Xi\right),$
we denote by 
\begin{align*}
H(\nu\,|\,\theta) & :=\sum_{r\in\Omega_{1}}\nu(\{r\})\log\frac{\nu(\{r\})}{\theta(\{r\})}
\end{align*}
the relative entropy of $\nu$ with respect to $\theta$. Notice that
this function $H\left(\cdot\,|\,\theta\right):\mathcal{M}_{1}\left(\Xi\right)\rightarrow\IR\cup\left\{ \infty\right\} $
is non-negative and concave as a function of $\nu$, and $\theta$
is its only minimum \cite[Definition 2.1.5 and the following remark]{DembZeit1998}.
\begin{lem}
\label{lemep} Let $\varepsilon\in(0,1)$ and suppose that $\nu,\theta\in\mathcal{M}_{1}\left(\Xi\right)$
are such that 
\begin{align}
|\nu(I)-\theta(I)| & \geq\varepsilon.\label{ptv1}
\end{align}

It follows that 
\begin{align}
H(\nu\,|\,\theta) & \geq\frac{1}{8}\varepsilon^{2}.\label{ptv2}
\end{align}
\end{lem}

\begin{proof}
Set $r=\theta(\{-1\})$ and $t=\nu(\{-1\})$. Inequality \eqref{ptv1}
implies that 
\begin{align}
|-t+(1-t)-(-r+(1-r))| & =2|t-r|\geq\varepsilon.\label{ptv3}
\end{align}
Moreover, we have that 
\begin{align}
f(t) & :=H(\nu\,|\,\theta)=t\log\left(\frac{t}{r}\right)+(1-t)\log\left(\frac{1-t}{1-r}\right).\label{ptv4}
\end{align}
As stated above, the minimum of $H(\nu\,|\,\theta)$ as a function
of $\nu$ is $\nu=\theta$. It follows that $f$ attains its minimum
at $r$, and $f(r)=0$. A first order Taylor series around $r$ implies
that 
\begin{align}
f(t) & =\frac{1}{2}f''(\xi)(t-r)^{2}=\frac{1}{2}\left(\frac{1}{\xi}+\frac{1}{1-\xi}\right)(t-r)^{2}\geq\frac{1}{2}(t-r)^{2},\label{ptv5}
\end{align}
where $\xi$ lies between $r$ and $t$, and therefore $\xi\in(0,1)$.
Displays \eqref{ptv3}-\eqref{ptv5} imply the desired result.
\end{proof}
We need a very precise upper bound for the tail probabilities under
$\boldsymbol{\mu}_{i,j}$. Therefore, we will use results from the
proof of Sanov's Theorem to establish said upper bound instead of
taking the statement from the theorem itself. The result given in
the theorem is in terms of the $\limsup$ of a sequence in the number
of observations $n$, but we will use an upper bound to be found in
\eqref{Sanov} which allows us to establish precisely how large $n$
has to be in order to have an upper bound on the probability of atypical
observations.
\begin{prop}
\label{PropSanov} For every $n\in\mathbb{N}$ and every $\varepsilon\in(0,1)$,
we have the upper bound
\begin{align*}
\boldsymbol{\mu}_{i,j}\left(\mathcal{O}_{i,j}^{n}(\varepsilon)\right) & \leq(n+1)^{2}e^{-\frac{1}{8}\varepsilon^{2}n}.
\end{align*}
\end{prop}

\begin{proof}
The result follows by the proof of Sanov's Theorem \cite[Theorem 2.1.10]{DembZeit1998}.
We use Eq. (2.1.12) in \cite{DembZeit1998} which in our case reads
as the following equation (we denote by $P_{\boldsymbol{\mu}_{i,j}}$
the probability induced by $\boldsymbol{\mu}_{i,j}$): 
\begin{align}
P_{\boldsymbol{\mu}_{i,j}}\left(\left\{ L_{n}^{Y_{i,j}(\omega)}\in\mathcal{O}_{i,j}(\varepsilon)\right\} \right) & =\boldsymbol{\mu}_{i,j}\left(\mathcal{O}_{i,j}^{n}(\varepsilon)\right)\leq(n+1)^{2}e^{-n\inf_{\eta\in\mathcal{O}_{i,j}(\varepsilon)\cup\mathcal{L}_{n}}H(\eta\,|\,\boldsymbol{\mu}_{i,j})},\label{Sanov}
\end{align}
where $\mathcal{L}_{n}=\left\{ \eta\,|\,\eta=L_{n}^{y},\:\text{for some }n\textup{ and }y\right\} $,
see above Lemma 2.1.12 in \cite{DembZeit1998}. The result follows
by \eqref{Sanov}, the definition of $\mathcal{O}_{i,j}(\varepsilon)$
and Lemma \ref{lemep}, which implies 
\[
H(\eta|\boldsymbol{\mu}_{i,j})\geq\frac{1}{8}\varepsilon^{2},
\]
for every $\eta\in\mathcal{O}_{i,j}(\varepsilon).$
\end{proof}

\subsection{\label{subsec:High-Temp}High Temperature Group Identification}
\begin{lem}
\label{CauchyS}Let $A$ be a positive definite $M\times M$ matrix.
Then
\[
A_{l,m}<\max\left\{ A_{l,l},A_{m,m}\right\} ,\quad l,m\in\IN_{M},l\neq m
\]
holds.
\end{lem}

\begin{proof}
Let $e_{l}$, $l\in\IN_{M}$, stand for the $l$-th vector of the
canonical basis of $\IR^{M}$. Since $A$ is positive definite, it
has a square root $A^{1/2}$ which is itself positive definite. We
have for all $l\neq m$
\begin{align*}
\langle Ae_{l},e_{m}\rangle & =\langle A^{1/2}e_{l},A^{1/2}e_{m}\rangle<\left\Vert A^{1/2}e_{l}\right\Vert \left\Vert A^{1/2}e_{m}\right\Vert \\
 & =\langle Ae_{l},e_{l}\rangle^{1/2}\langle Ae_{m},e_{m}\rangle^{1/2}\leq\left[\max\left\{ \langle Ae_{l},e_{l}\rangle^{1/2},\langle Ae_{m},e_{m}\rangle^{1/2}\right\} \right]^{2}.
\end{align*}
The strict inequality above follows from the Cauchy-Schwarz inequality,
since the invertibility of $A^{1/2}$ implies $A^{1/2}e_{l}$ and
$A^{1/2}e_{m}$ are not linearly dependent, and thus the inequality
is strict. As $\langle Ae_{l},e_{m}\rangle=A_{l,m}$, we conclude
that 
\begin{align*}
A_{l,m} & <\max\left\{ A_{l,l},A_{m,m}\right\} .
\end{align*}
\end{proof}
Recall that $H:=J^{-1}-I>0$ holds in the high temperature regime.
The previous Lemma implies
\begin{cor}
\label{cor:CS_H}We have for all $l,m\in\IN_{M}$ with $l\neq m$
\[
H_{l,m}^{-1}<\max\left\{ H_{l,l}^{-1},H_{m,m}^{-1}\right\} .
\]
\end{cor}

Using this corollary, we define the constants $\boldsymbol{\delta}>0$
and $N_{\textup{high}}\in\IN$, which are used in Theorem \ref{HighIdent}.
\begin{defn}
\label{DefinitionNd}Let
\begin{align*}
\eta & :=\min_{l\ne m}\left\{ \max\left\{ H_{l.l}^{-1},H_{m,m}^{-1}\right\} -H_{l,m}^{-1}\right\} ,\\
\xi & :=\min_{l\ne m}\left\{ \left.\left|H_{l.l}^{-1}-H_{m,m}^{-1}\right|\;\right|\;H_{l.l}^{-1}\neq H_{m,m}^{-1}\right\} .
\end{align*}
We define
\begin{align}
\boldsymbol{\delta} & :=\min\left\{ \eta,\xi\right\} ,\label{bdelta}
\end{align}
which is a strictly positive constant by Corollary \ref{cor:CS_H},
and 
\begin{equation}
\mathcal{O}^{n}:=\bigcup_{i\ne j}\mathcal{O}_{i,j}^{n}\left(\frac{1}{8N}\boldsymbol{\delta}\right).\label{eq:O^n}
\end{equation}

Moreover, we choose a fixed natural number $N_{{\rm high}}$ satisfying
(cf. Theorem \ref{thm:high_temp}) 
\begin{align}
\boldsymbol{C}_{\textup{high}}\frac{(\ln N)^{(6+M)/2}}{N^{2}} & \leq\frac{1}{8}\frac{1}{N}\boldsymbol{\delta},\label{pp1}
\end{align}
for every $N\geq N_{{\rm high}}$. 
\end{defn}

\begin{rem}
\label{rem:union_bound}Notice that Proposition \ref{PropSanov} implies
\[
\mathbb{P}\left(\left\{ \omega\in\mathcal{O}^{n}\right\} \right)\leq\sum_{i\neq j}\boldsymbol{\mu}_{i,j}\left(\mathcal{O}_{i,j}^{n}\left(\frac{1}{8N}\boldsymbol{\delta}\right)\right)\leq N^{2}(n+1)^{2}e^{-\frac{1}{8}\left(\frac{1}{8N}\boldsymbol{\delta}\right)^{2}n}.
\]
\end{rem}

In the following, we use the standard notation in which a random sample
(or an observation) is denoted by lowercase letters corresponding
to the random variables represented by uppercase letters. More precisely,
for a $\omega\in\Omega$, we denote by 
\begin{align}
x^{(t)} & =X^{(t)}(\omega)=\omega^{(t)}\in\Omega_{1}=\left\{ -1,1\right\} ^{N}.\label{stat}
\end{align}
We fix $\omega\notin\mathcal{O}^{n}$ for the rest of this section.
Given $\omega$, we obtain a sample $x^{(1)},x^{(2)},\ldots,x^{(n)}$
and calculate the empirical correlations 
\begin{equation}
\bar{y}_{i,j}:=\bar{y}_{i,j}(\omega):=\frac{1}{n}\sum_{t=1}^{n}x_{i}^{(t)}x_{j}^{(t)}\label{eq:emp_corr}
\end{equation}
for all $i,j\in\IN_{N}$ with $i\neq j$.
\begin{lem}
\label{High}Assume that $N\geq N_{{\rm high}}$. Let $n\in\mathbb{N}$.
Then we have 
\begin{align}
\left|\bar{y}_{i,j}-\frac{H_{\iota(i),\iota(j)}^{-1}}{N}\right| & \leq\frac{1}{4}\frac{1}{N}\boldsymbol{\delta}\quad i,j\in\IN_{N},i\neq j.\label{yaa6}
\end{align}
\end{lem}

\begin{proof}
Recall $L_{n}^{Y_{i,j}(\omega)}(I)=\bar{y}_{i,j}$ and $\mu_{i,j}(I)=\mathbb{E}\left(X_{\iota(i)}X_{\iota(j)}\right)$.
It follows that

\begin{align*}
\left|\bar{y}_{i,j}-\mathbb{E}\left(X_{\iota(i)}X_{\iota(j)}\right)\right|=\left|L_{n}^{Y_{i,j}(\omega)}(I)-\mu_{i,j}(I)\right| & \leq\frac{1}{8}\frac{1}{N}\boldsymbol{\delta}.
\end{align*}

Theorem \ref{thm:high_temp} and \eqref{pp1} imply
\begin{align*}
\left|\mathbb{E}\left(X_{\iota(i)}X_{\iota(j)}\right)-\frac{H_{\iota(i),\iota(j)}^{-1}}{N}\right| & \leq\frac{1}{8}\frac{1}{N}\boldsymbol{\delta},
\end{align*}
so we obtain 
\begin{align*}
\left|\bar{y}_{i,j}-\frac{H_{\iota(i),\iota(j)}^{-1}}{N}\right| & \leq\frac{1}{4}\frac{1}{N}\boldsymbol{\delta}.
\end{align*}
\end{proof}

Let $h_{(1)},h_{(2)},\ldots,h_{(q)}$ be an enumeration of the set
$\left\{ H_{l,l}^{-1}\right\} _{l\in\IN_{M}}$ with
\begin{align*}
h_{(1)} & >h_{(2)}>\ldots>h_{(q)}.
\end{align*}
We define a partition of the set of all groups $\IN_{M}$ based on
the values $h_{(u)}$ above.
\begin{defn}
\label{def:P_u}Let $P_{1},\ldots,P_{q}$ be a partition of $\IN_{M}$
induced by
\[
l\in P_{u}\quad\iff\quad H_{l,l}^{-1}=h_{(u)},\quad u\in\IN_{q}.
\]
\end{defn}

Next we recursively define two set sequences which will be employed
in the reconstruction algorithm of the group structure of the model.
\begin{defn}
\label{def:G_A}Let $A:=\left\{ \left(i,j\right)\in\IN_{N}\times\IN_{N}\,|\,i\neq j\right\} $.
We define
\begin{align*}
G_{1} & :=\left\{ (i,j)\in A\;\left|\;\left|\bar{y}_{i,j}-\frac{h_{(1)}}{N}\right|\leq\frac{1}{4N}\boldsymbol{\delta}\right.\right\} 
\end{align*}
and
\begin{align*}
A_{1} & :=\Big\{(i,j)\in A\;\left|\;\exists r\text{ such that }(i,r)\in G_{1}\;\textup{or}\;(r,j)\in G_{1}\right.\Big\}.
\end{align*}
Assume $G_{1},\ldots,G_{u}$ and $A_{1},\ldots A_{u}$ have been defined
for some $1\leq u<q$. Set
\begin{align*}
G_{u+1} & :=\left\{ (i,j)\in A\setminus A_{u}\;\left|\;\left|\bar{y}_{i,j}-\frac{h_{(u+1)}}{N}\right|\leq\frac{1}{4N}\boldsymbol{\delta}\right.\right\} 
\end{align*}
and
\begin{align*}
A_{u+1} & :=\left\{ \left.(i,j)\in A\;\right|\;\exists v\leq u+1,\exists r\text{ such that }(i,r)\in G_{v}\;\textup{or}\;(r,j)\in G_{v}\right\} .
\end{align*}
\end{defn}

\begin{rem}
It is a direct consequence of Definition \ref{def:G_A} that
\begin{align*}
G_{u}\cap G_{v} & =\emptyset,\quad u,v\in\IN_{q},u\neq v,\\
A_{u} & \subset A_{u+1},\quad u\in\IN_{q-1}.
\end{align*}
\end{rem}

In what follows, we will state and prove a series of lemmas, which
taken together will allow us to construct an equivalence relation
on $\IN_{M}$ based only on information contained in the sample of
observations, and more specifically the empirical correlations $\bar{y}_{i,j}$.
This equivalence relation corresponds to the partition into the sets
$\iota^{-1}(l)$, $l\in\IN_{M}$, and thus constitutes the solution
to the community detection problem stated in Theorem \ref{HighIdent}.
\begin{lem}
\label{lem:induction_base}Let $(i,j)\in A$. Then we have
\[
(i,j)\in G_{1}\quad\iff\quad\iota(i)=\iota(j)\in P_{1}.
\]
\end{lem}

\begin{proof}
Recall that we fixed $\omega\notin\mathcal{O}^{n}$ and hence $\bar{y}_{i,j}$
is fixed as well. Suppose $(i,j)\in G_{1}$. By Definition \ref{def:G_A},
\[
\left|\bar{y}_{i,j}-\frac{h_{(1)}}{N}\right|\leq\frac{1}{4N}\boldsymbol{\delta}
\]
holds. Lemma \ref{High} supplies the upper bound
\[
\left|\bar{y}_{i,j}-\frac{H_{\iota(i),\iota(j)}^{-1}}{N}\right|\leq\frac{1}{4N}\boldsymbol{\delta}.
\]
The two inequalities yield
\begin{equation}
\left|\frac{H_{\iota(i),\iota(j)}^{-1}}{N}-\frac{h_{(1)}}{N}\right|\leq\frac{1}{2N}\boldsymbol{\delta}.\label{eq:Hh}
\end{equation}
Note that $H_{k,l}^{-1}=h_{(1)}$ holds if and only if $k=l\in P_{1}$.
By Definition \ref{DefinitionNd},
\[
H_{k,l}^{-1}\leq h_{(1)}-\boldsymbol{\delta},\quad k\neq l,
\]
and
\[
H_{k,k}^{-1}\leq h_{(1)}-\boldsymbol{\delta},\quad k\notin P_{1}.
\]
The above and inequality \eqref{eq:Hh} give
\[
\iota(i)=\iota(j)\in P_{1}.
\]
Thus, we have shown that $(i,j)\in G_{1}$ implies $\iota(i)=\iota(j)\in P_{1}$.
The converse is proved analogously.
\end{proof}
\begin{lem}
\label{lem:step_RL}Let $u\in\IN_{q}$. Suppose we have
\[
(i,j)\in G_{v}\quad\iff\quad\iota(i)=\iota(j)\in P_{v}
\]
for all $(i,j)\in A$ and all $1\leq v<u$. Then
\[
\iota(i)=\iota(j)\in P_{u}\quad\implies\quad(i,j)\in G_{u}
\]
holds for all $(i,j)\in A$.
\end{lem}

\begin{proof}
Let $\iota(i)=\iota(j)\in P_{u}$. To obtain a contradiction, suppose
$\left(i,j\right)\in A_{u-1}$. Then there is some $r\neq i$ such
that $\left(i,r\right)\in G_{v}$ or $\left(i,r\right)\in G_{v}$
for some $v<u$. Without loss of generality, suppose $\left(i,r\right)\in G_{v}$.
By assumption, this implies $\iota(i)=\iota(r)\in P_{v}$. However,
this contradicts $\iota(i)\in P_{u}$ because of $P_{v}\cap P_{u}=\emptyset$.
Therefore, $\left(i,j\right)\notin A_{u-1}$ must hold. Since $\omega\notin\mathcal{O}^{n}$,
Lemma \ref{High} yields
\[
\left|\bar{y}_{i,j}-\frac{h_{(u)}}{N}\right|\leq\frac{1}{4N}\boldsymbol{\delta},
\]
and thus $(i,j)\in G_{u}$ holds by definition of $G_{u}$:
\[
G_{u}:=\left\{ (i,j)\in A\setminus A_{u-1}\;\left|\;\left|\bar{y}_{i,j}-\frac{h_{(u)}}{N}\right|\leq\frac{1}{4N}\boldsymbol{\delta}\right.\right\} .
\]
\end{proof}
\begin{lem}
Let $u\in\IN_{q}$. Suppose we have
\[
(i,j)\in G_{v}\quad\iff\quad\iota(i)=\iota(j)\in P_{v}
\]
for all $(i,j)\in A$ and all $1\leq v<u$. Then
\[
(i,j)\in G_{u}\quad\implies\quad\iota(i),\iota(j)\notin P_{v},\quad v<u.
\]
\end{lem}

\begin{proof}
Assume $(i,j)\in G_{u}$. By Definition \ref{def:G_A}, there is no
$r\in\IN_{N}$ such that
\begin{equation}
\left(i,r\right)\in G_{v}\quad\textup{or}\quad\left(r,j\right)\in G_{v}\label{eq:G_v}
\end{equation}
for some $v<u$, as the latter would imply $\left(i,j\right)\in A_{v}$
and thus $\left(i,j\right)\notin G_{u}$ in contradiction to our assumption.

To obtain a contradiction, assume that $\iota(i)\in P_{v}$ for some
$v<u$. Since each group has at least two individuals (cf. Remark
\ref{rem:alpha}), there is an $r\neq i$ with $\iota(r)=\iota(i)\in P_{v}$.
By assumption, this implies $(i,r)\in G_{v}$ as $v<u$. This contradicts
\eqref{eq:G_v}, and hence we conclude that $\iota(i)\notin P_{v}$.
$\iota(j)\notin P_{v}$ follows by the same arguments.
\end{proof}
\begin{lem}
\label{lem:step_LR}Let $u\in\IN_{q}$. Suppose we have
\[
(i,j)\in G_{v}\quad\iff\quad\iota(i)=\iota(j)\in P_{v}
\]
for all $(i,j)\in A$ and all $1\leq v<u$. Then
\[
(i,j)\in G_{u}\quad\implies\quad\iota(i)=\iota(j)\in P_{u}
\]
holds for all $(i,j)\in A$.
\end{lem}

\begin{proof}
Let $(i,j)\in G_{u}$. By the last lemma, we have $\iota(i),\iota(j)\notin P_{v}$
for all $v<u$. As $(i,j)\in G_{u}$, we have
\begin{equation}
\left|\bar{y}_{i,j}-\frac{h_{(u)}}{N}\right|\leq\frac{1}{4N}\boldsymbol{\delta},\label{eq:yh}
\end{equation}
To obtain a contradiction, assume $\iota(i)\neq\iota(j)$. Then we
have
\[
H_{\iota(i),\iota(j)}^{-1}\leq h_{(u)}-\boldsymbol{\delta}
\]
by Definition \ref{DefinitionNd}, which contradicts \eqref{eq:yh}.
Therefore, $\iota(i)=\iota(j)$ must hold. Next assume $\iota(i)=\iota(j)\in P_{w}$
for some $w>u$. Once again, Definition \ref{DefinitionNd} implies
\[
H_{\iota(i),\iota(j)}^{-1}=h_{(w)}\leq h_{(u)}-\boldsymbol{\delta},
\]
and we have a contradiction to \eqref{eq:yh}. Hence, $\iota(i)=\iota(j)\in P_{u}$
is proved.
\end{proof}
\begin{cor}
\label{cor:induction_step}Let $u\in\IN_{q}$. Suppose we have
\[
(i,j)\in G_{v}\quad\iff\quad\iota(i)=\iota(j)\in P_{v}
\]
for all $(i,j)\in A$ and all $1\leq v<u$. Then
\[
\iota(i)=\iota(j)\in P_{u}\quad\iff\quad(i,j)\in G_{u}
\]
holds for all $(i,j)\in A$.
\end{cor}

\begin{proof}
The statement follows from Lemmas \ref{lem:step_RL} and \ref{lem:step_LR}.
\end{proof}
\begin{prop}
\label{prop:Pu_Gu}For all $(i,j)\in A$ and all $u\in\IN_{q}$, we
have
\[
\iota(i)=\iota(j)\in P_{u}\quad\iff\quad(i,j)\in G_{u}.
\]
\end{prop}

\begin{proof}
Together, Lemma \ref{lem:induction_base} and Corollary \ref{cor:induction_step}
constitute a proof by induction over $u\in\IN_{q}$ of the proposition's
statement.
\end{proof}
\begin{prop}
\label{prop:equiv}Let the relation $\sim$ on $\IN_{N}$ be defined
by
\[
i\sim j\quad\iff\quad i=j\quad\textup{or}\quad\exists u\in\IN_{q}:\;(i,j)\in G_{u},\quad i,j\in\IN_{N}.
\]
Then $\sim$ is an equivalence relation, and, for all $i,j\in\IN_{N}$,
$i\sim j$ holds if and only if $\iota(i)=\iota(j)$.
\end{prop}

\begin{proof}
By Proposition \ref{prop:Pu_Gu}, for all $\left(i,j\right)\in A$,
the existence of a $u\in\IN_{q}$ such that $\left(i,j\right)\in G_{u}$
is equivalent to $\iota(i)=\iota(j)\in P_{u}$. The function $\iota:\IN_{N}\rightarrow\IN_{M}$
yields a partition of $\IN_{N}$ with the equivalence classes given
by $\iota^{-1}(l)$, $l\in\IN_{M}$, i.e. the indices $i\in\IN_{N}$
belonging to group $l$. Therefore, $\left(i,j\right)\in G_{u}$ is
equivalent to the statement that $i$ and $j$ belong to the same
equivalence class given by some group $l\in\IN_{M}$.
\end{proof}
For the convenience of the reader, we restate Theorem \ref{HighIdent}.
The theorem says that the problem of community detection has a solution
which can be identified with high probability.
\begin{thm}
[Group Identification and Reconstruction Procedure in the High Temperature Regime]Let
$\boldsymbol{\delta}$ and $N_{\textup{high}}$ be as specified in
Definition \ref{DefinitionNd}. Suppose that $N\geq N_{{\rm high}}$
and $n$ is a natural number. A sample of $n$ observations $x^{(1)},x^{(2)},\ldots,x^{(n)}$
allows us to recover the group partition of $\mathbb{N}_{N}$ with
high probability. The probability that we cannot recover the group
partition is bounded above by the exponentially decaying function
\begin{align*}
e_{\textup{high}}(n) & :=N^{2}(n+1)^{2}e^{-\frac{1}{8}\left(\frac{1}{8N}\boldsymbol{\delta}\right)^{2}n}.
\end{align*}
\end{thm}

\begin{proof}
Recall the definition of $\mathcal{O}^{n}$ in \eqref{eq:O^n} and
that we have fixed $\omega\notin\mathcal{O}^{n}$. By Proposition
\ref{prop:equiv}, the set sequence $G_{u}$, $u\in\IN_{q}$, allows
us to define an equivalence relation $\sim$ on $\IN_{N}$ which corresponds
to the group structure of the model in the sense that for all $i,j\in\IN_{N}$
$i\sim j$ holds if and only if $i$ and $j$ belong to the same group.
Then the probability that recovery is impossible is bounded above
by $\IP\left(\left\{ \omega\in\mathcal{O}^{n}\right\} \right)\leq e_{\textup{high}}(n)$
according to Remark \ref{rem:union_bound}. This concludes the proof
of Theorem \ref{HighIdent}.
\end{proof}

\subsection{\label{subsec:Low-Temp}Low Temperature Group Identification}

For the next definition, recall \eqref{eq:Zl}:
\begin{align*}
\vec{Z}_{l} & :=\frac{1}{\sqrt{\sum_{k}\frac{1}{\det(H_{k})^{1/2}}}}\begin{pmatrix}\frac{1}{\det(H_{1})^{1/4}}\tanh\left(z_{l}^{(1)}\right)\\
\frac{1}{\det(H_{2})^{1/4}}\tanh\left(z_{l}^{(2)}\right)\\
\vdots\\
\frac{1}{\det(H_{2})^{1/4}}\tanh\left(z_{l}^{(K)}\right)
\end{pmatrix}\in\IR^{K},\quad l\in\IN_{M},
\end{align*}
the notation $\left\{ z^{(k)}\right\} _{k\in\IN_{K}}$ for the minima
of the function $F$, and Remark \ref{rem:Zl}, and recall that we
defined the subsets of $\Omega$
\[
\mathcal{O}_{i,j}^{n}(\varepsilon):=\left\{ \omega\in\Omega\;\left|\;L_{n}^{Y_{i,j}(\omega)}\in\mathcal{O}_{i,j}(\varepsilon)\right.\right\} ,\quad\varepsilon>0.
\]

\begin{defn}
\label{DefinitionNdl} We set
\begin{align*}
\eta & :=\min\left\{ \left.\left|\left\Vert \vec{Z}_{l}\right\Vert ^{2}-\left\Vert \vec{Z}_{m}\right\Vert ^{2}\right|\;\right|\;\left\Vert \vec{Z}_{l}\right\Vert \ne\left\Vert \vec{Z}_{m}\right\Vert ,l\neq m\right\} ,\\
\xi & :=\min\left\{ \left.\left|\left\Vert \vec{Z}_{l}\right\Vert ^{2}-\left\langle \vec{Z}_{l},\vec{Z}_{m}\right\rangle \right|\;\right|\;\left\Vert \vec{Z}_{l}\right\Vert =\left\Vert \vec{Z}_{m}\right\Vert ,l\neq m\right\} ,\\
\boldsymbol{\boldsymbol{\gamma}} & :=\min\left\{ \eta,\xi\right\} 
\end{align*}
(note that $\boldsymbol{\boldsymbol{\gamma}}$ is positive by Remark
\ref{rem:Zl}), and
\begin{align}
\mathcal{U}^{n} & :=\bigcup_{i\ne j}\mathcal{O}_{i,j}^{n}\left(\frac{1}{8}\boldsymbol{\boldsymbol{\gamma}}\right).\label{calu}
\end{align}

Moreover, we choose a fixed natural number $N_{{\rm low}}$ satisfying
(cf. Theorem \ref{thm:low_temp}) 
\begin{align}
\boldsymbol{C}_{\textup{low}}\frac{(\ln N)^{(M+3)/2}}{N^{1/2}} & \leq\frac{1}{8}\boldsymbol{\boldsymbol{\boldsymbol{\gamma}}}\label{pp1l}
\end{align}
for every $N\geq N_{{\rm low}}$. 
\end{defn}

\begin{rem}
\label{rem:union_bound_low}Notice that Proposition \ref{PropSanov}
implies that
\[
\mathbb{P}\left(\left\{ \omega\in\mathcal{U}^{n}\right\} \right)\leq\sum_{i,j}\boldsymbol{\mu}_{i,j}\left(\mathcal{O}_{i,j}^{n}\left(\frac{1}{8}\boldsymbol{\boldsymbol{\gamma}}\right)\right)\leq N^{2}(n+1)^{2}e^{-\frac{1}{8}\big(\frac{1}{8}\boldsymbol{\gamma}\big)^{2}n}.
\]
\end{rem}

Recall that Theorem \ref{thm:low_temp} states
\begin{align*}
\left|\mathbb{E}\left(X_{i}X_{j}\right)-\left\langle \vec{Z}_{\iota(i)},\vec{Z}_{\iota(j)}\right\rangle \right| & \leq\boldsymbol{C}_{\textup{low}}\frac{\left(\ln N\right)^{\left(M+3\right)/2}}{\sqrt{N}},\quad i,j\in\IN_{N},i\neq j.
\end{align*}

We fix $\omega\notin\mathcal{U}^{n}$ for the rest of this section.
Given $\omega$, we obtain a sample $x^{(1)},x^{(2)},\ldots,x^{(n)}$.
Recall the definition of the empirical correlations 
\[
\bar{y}_{i,j}:=\bar{y}_{i,j}(\omega):=\frac{1}{n}\sum_{t=1}^{n}x_{i}^{(t)}x_{j}^{(t)}
\]
for all $i,j\in\IN_{N}$ with $i\neq j$.
\begin{lem}
\label{Low}Assume that $N\geq N_{{\rm Low}}$. Set $n\in\mathbb{N}$.
Then the following inequality holds: 
\begin{align}
\left|\bar{y}_{i,j}-\left\langle \vec{Z}_{\iota(i)},\vec{Z}_{\iota(j)}\right\rangle \right| & \leq\frac{1}{4}\boldsymbol{\gamma},\quad i\neq j.\label{yaa6l}
\end{align}
\end{lem}

\begin{proof}
Since $\omega\in\left(\mathcal{U}^{n}\right)^{c}$, it follows that
(recall that $L_{n}^{Y_{i,j}(\omega)}(I)=\bar{y}_{i,j}$ and $\boldsymbol{\mu}_{i,j}(I)=\mathbb{E}(X_{\iota(i)}X_{\iota(j)})$
and \eqref{calu}, \eqref{calo})

\begin{align}
\left|L_{n}^{Y_{i,j}(\omega)}(I)-\mu_{i,j}(I)\right|=\left|\bar{y}_{i,j}-\mathbb{E}(X_{\iota(i)}X_{\iota(j)})\right| & \leq\frac{1}{8}\boldsymbol{\gamma},\label{PI1}
\end{align}
for all $i\neq j$. Recalling Theorem \ref{thm:low_temp},
\begin{align*}
\left\Vert \mathbb{E}(X_{i}X_{j})-\left\langle \vec{Z}_{\iota(i)},\vec{Z}_{\iota(j)}\right\rangle \right\Vert  & \leq\boldsymbol{C}_{\textup{low}}\frac{\left(\ln N\right)^{\left(M+3\right)/2}}{\sqrt{N}}\leq\frac{1}{8}\boldsymbol{\gamma},
\end{align*}
we obtain
\begin{align}
\left|\bar{y}_{i,j}-\left\langle \vec{Z}_{\iota(i)},\vec{Z}_{\iota(j)}\right\rangle \right| & \leq\frac{1}{4}\boldsymbol{\gamma}\label{yaa7l}
\end{align}
for every $i\neq j$.
\end{proof}
Let $g_{(1)},g_{(2)},\ldots,g_{(q)}$ be an enumeration of the set
$\left\{ \left\Vert \vec{Z}_{l}\right\Vert ^{2}\right\} _{l\in\IN_{M}}$
with
\begin{align*}
g_{(1)} & >g_{(2)}>\ldots>g_{(q)}.
\end{align*}
We define a partition of the set of all groups $\IN_{M}$ based on
the values $g_{(u)}$ above.
\begin{defn}
\label{def:P_u-1}Let $P_{1},\ldots,P_{q}$ be a partition of $\IN_{M}$
induced by
\[
l\in P_{u}\quad\iff\quad\left\Vert \vec{Z}_{l}\right\Vert ^{2}=g_{(u)},\quad u\in\IN_{q}.
\]
\end{defn}

Next we recursively define two set sequences which will be employed
in the reconstruction algorithm of the group structure of the model.
\begin{defn}
\label{def:G_A-1}Let $A:=\left\{ \left(i,j\right)\in\IN_{N}\times\IN_{N}\,|\,i\neq j\right\} $.
We define
\begin{align*}
G_{1} & :=\left\{ (i,j)\in A\;\left|\;\left|\bar{y}_{i,j}-g_{(1)}\right|\leq\frac{1}{4}\boldsymbol{\gamma}\right.\right\} 
\end{align*}
and
\begin{align*}
A_{1} & :=\Big\{(i,j)\in A\;\left|\;\exists r\text{ such that }(i,r)\in G_{1}\;\textup{or}\;(r,j)\in G_{1}\right.\Big\}.
\end{align*}
Assume $G_{1},\ldots,G_{u}$ and $A_{1},\ldots A_{u}$ have been defined
for some $1\leq u<q$. Set
\begin{align*}
G_{u+1} & :=\left\{ (i,j)\in A\setminus A_{u}\;\left|\;\left|\bar{y}_{i,j}-g_{(u+1)}\right|\leq\frac{1}{4}\boldsymbol{\gamma}\right.\right\} 
\end{align*}
and
\begin{align*}
A_{u+1} & :=\left\{ \left.(i,j)\in A\;\right|\;\exists v\leq u+1,\exists r\text{ such that }(i,r)\in G_{v}\;\textup{or}\;(r,j)\in G_{v}\right\} .
\end{align*}
\end{defn}

\begin{rem}
It is a direct consequence of Definition \ref{def:G_A-1} that
\begin{align*}
G_{u}\cap G_{v} & =\emptyset,\quad u,v\in\IN_{q},u\neq v,\\
A_{u} & \subset A_{u+1},\quad u\in\IN_{q-1}.
\end{align*}
\end{rem}

The proof of the next proposition follows the same steps as that
of Proposition \ref{prop:equiv}, using Lemma \ref{Low} instead of
Lemma \ref{High}.
\begin{prop}
\label{prop:equiv-1}Let the relation $\sim$ on $\IN_{N}$ be defined
by
\[
i\sim j\quad\iff\quad i=j\quad\textup{or}\quad\exists u\in\IN_{q}:\;(i,j)\in G_{u},\quad i,j\in\IN_{N}.
\]
Then $\sim$ is an equivalence relation, and, for all $i,j\in\IN_{N}$,
$i\sim j$ holds if and only if $\iota(i)=\iota(j)$.
\end{prop}

\begin{thm}
[Group Identification and Reconstruction Procedure in the Low Temperature Regime]Let
the model be in the low temperature regime and non-critical. Let $z^{(1)},\ldots,z^{(K)}\in\IR^{M}$
be the minima of the function $F$, and assume there are no $l$ and
$m$ with $l\neq m$ such that $z_{l}^{(k)}=z_{m}^{(k)}$ holds for
every $k=1,\ldots,K$. There is a fixed natural number $N_{{\rm low}}$
and a constant $\boldsymbol{\gamma}>0$ (cf. Definition \ref{DefinitionNdl})
such that if $N\geq N_{{\rm low}}$ and $n\in\IN$, a sample of $n$
observations $x^{(1)},x^{(2)},\ldots,x^{(n)}$ allows us to recover
the group partition with high probability. The probability that we
cannot recover it is bounded above by the exponentially decaying function
\begin{align*}
e_{{\rm low}}(n) & :=N^{2}(n+1)^{2}e^{-\frac{1}{8}\left(\frac{1}{8}\boldsymbol{\gamma}\right)^{2}n}.
\end{align*}
\end{thm}

\begin{proof}
Recall the definition of $\mathcal{U}^{n}$ in \eqref{calu} and
that we have fixed $\omega\notin\mathcal{U}^{n}$. By Proposition
\ref{prop:equiv-1}, the set sequence $G_{u}$, $u\in\IN_{q}$, allows
us to define an equivalence relation $\sim$ on $\IN_{N}$ which corresponds
to the group structure of the model in the sense that for all $i,j\in\IN_{N}$
$i\sim j$ holds if and only if $i$ and $j$ belong to the same group.
Then the probability that recovery is impossible is bounded above
by $\IP\left(\left\{ \omega\in\mathcal{U}^{n}\right\} \right)\leq e_{{\rm low}}(n)$
according to Remark \ref{rem:union_bound_low}. This concludes the
proof of Theorem \ref{LowIdent}.
\end{proof}

\section{\label{sec:Proof-thm_bounds}Proofs of Theorems \ref{thm:high_temp}
and \ref{thm:low_temp}}

Recall that 
\[
F(x)=\frac{1}{2}x^{T}J^{-1}x-\sum_{l=1}^{M}\alpha_{l}\ln\cosh\left(\frac{x_{l}}{\sqrt{\alpha_{l}}}\right),\quad x\in\IR^{M}.
\]
We define for fixed $i,j\in\IN_{N}$ with $i\neq j$
\begin{align}
Z_{2}(N) & :=Z_{2}(N;i,j):=\int_{\IR^{M}}e^{-NF(x)}\tanh\left(\frac{x_{\iota(i)}}{\sqrt{\alpha_{\iota(i)}}}\right)\tanh\left(\frac{x_{\iota(j)}}{\sqrt{\alpha_{\iota(j)}}}\right)\textup{d}x\label{Z2N}
\end{align}
and 
\begin{align}
Z_{0}(N) & :=\int_{\IR^{M}}e^{-NF(x)}\textup{d}x.\label{Z0N}
\end{align}

Then the representation
\begin{equation}
\mathbb{E}\left(X_{i}X_{j}\right)=\frac{Z_{2}(N)}{Z_{0}(N)}\label{Deltalam}
\end{equation}

holds by Theorem \ref{thm:de_Finetti_rep} in the Appendix.

\subsection{Proof of Theorem \ref{thm:high_temp}}
\begin{lem}
\label{lem1} In the high temperature regime, where $H=J^{-1}-I>0$,
$0$ is the only minimum of $F$ and the lower bound 
\[
F(x)\geq\frac{1}{2}\langle Hx,x\rangle\geq c_{H}\left\Vert x\right\Vert ^{2},
\]
holds, where $c_{H}>0$ is half of the minimum eigenvalue of $H.$ 
\end{lem}

\begin{proof}
Let $\mathbf{H}_{F}:\IR^{M}\rightarrow\IR$ be the Hessian matrix
of $F$ at $x\in\IR^{M}$. A direct calculation using hyperbolic trigonometric
identities shows that 
\begin{align*}
(\mathbf{H}_{F})_{l,m}(x) & =H_{l,m}+\delta_{l,m}\tanh^{2}\left(\frac{x_{l}}{\sqrt{\alpha_{l}}}\right),
\end{align*}
where $\delta_{l,m}$ is the Kronecker delta. This implies that $\mathbf{H}_{F}(x)$
is positive definite for all $x\in\IR^{M}$. Since $F'(0)=0$, $0$
is the only minimum of $F$, and, due to $F(0)=0$, $F$ is strictly
positive on the complement of $\{0\}$. Moreover, a Taylor series
expansion of the function $[0,1]\ni t\mapsto F(tx)$ in $t$ of order
1 with remainder of order 2 in integral form shows that 
\[
F(x)=\int_{0}^{1}\Big\langle(\mathbf{H}_{F})_{i,j}(tx)x,x\Big\rangle(1-t)\,\textup{d}t\geq\frac{1}{2}\langle Hx,x\rangle\geq c_{H}\left\Vert x\right\Vert ^{2},
\]
where in the above equation we used $F'(0)=0$ and the spectral theorem
which yields strict positivity of the at most $M$ eigenvalues of
the self-adjoint $H$.
\end{proof}
We restate Theorem \ref{thm:high_temp} for the convenience of the
reader:
\begin{thm}
In the high temperature regime, i.e. for $I-J>0$, set $H:=J^{-1}-I$.
Then there is a positive constant $\boldsymbol{C}_{\textup{high}}$
such that for all $N\in\IN$
\[
\left|\mathbb{E}\left(X_{i}X_{j}\right)-\frac{H_{\iota(i),\iota(j)}^{-1}}{N}\right|\leq\boldsymbol{C}_{\textup{high}}\frac{\left(\ln N\right)^{(6+M)/2}}{N^{2}}\quad i,j\in\IN_{N},i\neq j,
\]
holds, where we use the notation $H_{\iota(i),\iota(j)}^{-1}=\left(H^{-1}\right)_{\iota(i),\iota(j)}.$
\end{thm}

\begin{proof}
We will show the statement for $\iota(i)=\iota(j)$. The case $\iota(i)\neq\iota(j)$
can be handled analogously. For any set $A\subset\IR^{M}$, let $A^{c}:=\IR^{M}\backslash A$
be the complement. We will work with $Z_{2}(N)$ and $Z_{0}(N)$ defined
in \eqref{Z2N} and \eqref{Z0N}. We denote by
\begin{align*}
B(z,r) & :=\left\{ \left.x\in\mathbb{R}^{M}\;\right|\;\|x-z\|<r\right\} ,
\end{align*}
the ball with radius $r$, centred at $z$, and we set (see Lemma
\ref{lem1}) 
\begin{align*}
r_{N} & :=\sqrt{\frac{(M+4)\ln(N)}{c_{H}(N-1)}},
\end{align*}
which is selected in order to fulfil $e^{-(N-1)c_{H}r_{N}^{2}}\leq\frac{1}{N^{M+4}}$.
We split $Z_{2}(N)$ and $Z_{0}(N)$ into integrals over $B\left(0,r_{N}\right)$
and over $B\left(0,r_{N}\right)^{c}$. Lemma \ref{lem1} implies that
\begin{align}
\int_{B(0,r_{N})^{c}}e^{-NF(x)}\tanh^{2}\left(\frac{x_{\iota(i)}}{\sqrt{\alpha_{\iota(i)}}}\right)\,\textup{d}x & \leq e^{-\left(N-1\right)c_{H}r_{N}^{2}}\int_{B(0,r_{N})^{c}}e^{-F(x)}\tanh^{2}\left(\frac{x_{\iota(i)}}{\sqrt{\alpha_{\iota(i)}}}\right)\,\textup{d}x\nonumber \\
 & \leq\frac{1}{N^{M+4}}\int_{\mathbb{R}^{M}}e^{-c_{H}\left\Vert x\right\Vert ^{2}}\tanh^{2}\left(\frac{x_{\iota(i)}}{\sqrt{\alpha_{\iota(i)}}}\right)\,\textup{d}x\leq C\frac{1}{N^{M+4}},\label{to1}
\end{align}
and likewise we conclude that 
\begin{align}
\int_{B(r_{N},0)^{c}}e^{-NF_{J}(x)}\,\textup{d}x & \leq C\frac{1}{N^{M+4}}\label{to2}
\end{align}
and 
\begin{align}
\int_{B(0,r_{N})^{c}}e^{-N\langle Hx,x\rangle}\frac{x_{\iota(i)}^{2}}{\alpha_{\iota(i)}}\,\textup{d}x & \leq C\frac{1}{N^{M+4}},\hspace{1.5cm}\int_{B(0,r_{N})^{c}}e^{-N\langle Hx,x\rangle}\,\textup{d}x\leq C\frac{1}{N^{M+4}}.\label{to2.1}
\end{align}
Notice that $\frac{\textup{d}}{\textup{d}s}\ln\cosh(s)=\tanh(s)$
and $\frac{\textup{d}}{\textup{d}s}\tanh(s)=1-\tanh^{2}(s)$. This
implies that all even derivatives of $\ln\cosh(s)$ are polynomials
with even powers of $\tanh(s)$ and all odd derivatives of $\ln\cosh(s)$
are polynomials with odd powers of $\tanh(s)$ . Therefore, all such
derivatives are uniformly bounded on $\mathbb{R}$, and odd derivatives
evaluated at $0$ vanish (and the same holds for $F$). Taking this
into consideration, and using that $F(0)=0$ and the Hessian of $F$
at $0$ is $H$, the remainder formula in Taylor's Theorem leads to
\begin{align}
\left|F(x)-\frac{1}{2}\langle Hx,x\rangle\right| & \leq C\left\Vert x\right\Vert ^{4},\label{to3}
\end{align}
for some constant $C$, independent of $x$ (for $\left\Vert x\right\Vert \leq1$
we use Taylor's Theorem, for $\left\Vert x\right\Vert >1$ we use
that $F(x)$ and $\langle Hx,x\rangle$ grow quadratically in $\left\Vert x\right\Vert $).
Note that for positive $s$ and $t$, $s\leq t$, the mean value theorem
implies that $\left|e^{-s}-e^{-t}\right|=\left|e^{-s}\right|\left|e^{s-t}-1\right|\leq e^{-s}|s-t|\leq|s-t|$.
This, together with the fact that $\tanh$ is bounded and \eqref{to3},
implies that 
\begin{align}
\Big|\int_{B(0,r_{N})}e^{-NF(x)}\tanh^{2}\left(\frac{x_{\iota(i)}}{\sqrt{\alpha_{\iota(i)}}}\right)\,\textup{d}x & -\int_{B(0,r_{N})}e^{-N\frac{1}{2}\langle Hx,x\rangle}\tanh^{2}\left(\frac{x_{\iota(i)}}{\sqrt{\alpha_{\iota(i)}}}\right)\,\textup{d}x\Big|\nonumber \\
 & \leq CN\int_{B(0,r_{N})}\tanh^{2}\left(\frac{x_{\iota(i)}}{\sqrt{\alpha_{\iota(i)}}}\right)\left\Vert x\right\Vert ^{4}\,\textup{d}x\nonumber \\
 & \leq CN\int_{B(0,r_{N})}\left\Vert x\right\Vert ^{6}\,\textup{d}x\nonumber \\
 & \leq CNr_{N}^{6+M}\leq C\left(\ln N\right)^{(6+M)/2}\frac{1}{N^{(4+M)/2}}.\label{eq:to2a}
\end{align}
Likewise, we get 
\begin{align}
\Big|\int_{B(0,r_{N})}e^{-NF(x)}\textup{d}x & -\int_{B(0,r_{N})}e^{-N\frac{1}{2}\langle Hx,x\rangle}\textup{d}x\Big|\leq C\left(\ln N\right)^{(4+M)/2}\frac{1}{N^{(2+M)/2}}.\label{to2c}
\end{align}
The Taylor series of $\tanh$ at $0$ (up to second order, using that
$\tanh(0)=0=\tanh''(0)$ and that $\tanh'''$ is uniformly bounded),
implies that there is a constant $C$ such that 
\begin{align*}
|\tanh(s)-s| & \leq C|s|^{3},
\end{align*}
and this leads to
\begin{align}
\Big|\int_{B(0,r_{N})}e^{-N\frac{1}{2}\langle Hx,x\rangle}\tanh^{2}\left(\frac{x_{\iota(i)}}{\sqrt{\alpha_{\iota(i)}}}\right)\,\textup{d}x & -\int_{B(0,r_{N})}e^{-N\frac{1}{2}\langle Hx,x\rangle}\frac{x_{\iota(i)}^{2}}{\alpha_{\iota(i)}}\,\textup{d}x\Big|\nonumber \\
 & \leq C\int_{B(0,r_{N})}\left\Vert x\right\Vert ^{4}\,\textup{d}x\leq CNr_{N}^{4+M}\leq C\left(\ln N\right)^{(4+M)/2}\frac{1}{N^{(4+M)/2}}.\label{eq:to4}
\end{align}

Next, Proposition \ref{MGauss} implies that 
\begin{align}
\int_{\mathbb{R}^{M}}e^{-N\frac{1}{2}\langle Hx,x\rangle}\frac{x_{\iota(i)}^{2}}{\alpha_{\iota(i)}}\,\textup{d}x & =\frac{(2\pi)^{M/2}}{\det(H)^{1/2}}\frac{1}{N^{M/2+1}}H_{\lambda,\lambda}^{-1},\hspace{1cm}\int_{\mathbb{R}^{M}}e^{-N\frac{1}{2}\langle Hx,x\rangle}\,\textup{d}x=\frac{(2\pi)^{M/2}}{\det(H)^{1/2}}\frac{1}{N^{M/2}}.\label{gausint}
\end{align}
It follows from \eqref{Z2N}, \eqref{to1}, \eqref{to2.1},\eqref{eq:to2a},
\eqref{eq:to4}, and \eqref{gausint} that 
\begin{align}
\Big|Z_{2}(N)-\frac{(2\pi)^{M/2}}{\det(H)^{1/2}}\frac{1}{N^{M/2+1}}H_{\lambda,\lambda}^{-1}\Big| & \leq C\ln(N)^{(6+M)/2}\frac{1}{N^{(4+M)/2}}\label{cas1}
\end{align}
and it follows from \eqref{Z0N}, \eqref{to2}, \eqref{to2.1}, \eqref{to2c},
and \eqref{gausint} that 
\begin{align}
\Big|Z_{0}(N)-\frac{(2\pi)^{M/2}}{\det(H)^{1/2}}\frac{1}{N^{M/2}}\Big| & \leq C\ln(N)^{(4+M)/2}\frac{1}{N^{(2+M)/2}}.\label{cas2}
\end{align}
Set
\begin{align*}
A & :=\frac{(2\pi)^{M/2}}{\det(H)^{1/2}}\frac{1}{N^{M/2+1}}H_{\lambda,\lambda}^{-1},\\
B & :=\frac{(2\pi)^{M/2}}{\det(H)^{1/2}}\frac{1}{N^{M/2}}.
\end{align*}
Then, using \eqref{cas1} and \eqref{cas2}, we obtain

\begin{align}
\left|\frac{Z_{2}(N)}{Z_{0}(N)}-\frac{1}{N}H_{\lambda,\lambda}^{-1}\right| & =\left|\frac{Z_{2}(N)-A}{Z_{0}(N)}+A\left(\frac{1}{Z_{0}}-\frac{1}{B}\right)\right|\nonumber \\
 & \leq\left|\frac{Z_{2}(N)-A}{Z_{0}(N)}\right|+|A|\frac{|Z_{0}(N)-B|}{|Z_{0}(N)||B|}\leq C\ln(N)^{(6+M)/2}\frac{1}{N^{2}}.\label{eq:high_error}
\end{align}

Displays \eqref{Deltalam} and \eqref{eq:high_error} lead to the
desired result.
\end{proof}

\subsection{Proof of Theorem \ref{thm:low_temp}}

Recall the definition of $F$ given in \eqref{eq:F} and Lemma \ref{lem:minima}.
\begin{lem}
\label{LChin}There exist constants $\delta>0$ and $c>0$ such that
the following estimation holds: 
\begin{align}
F(x)-\alpha & \geq c\left\Vert x-z^{(k)}\right\Vert ^{2},\hspace{1cm}\text{for}\:\left\Vert x-z^{(j)}\right\Vert \geq\delta,j\ne k,\label{T60}
\end{align}
where $\alpha$ is the minimum value of $F$, and the balls $B\left(z^{(k)},\delta\right)$
are disjoint for $k\in\{1,\ldots,K\}$. Also, for every $k$ and every
$y\in\IR^{M}$, 
\begin{align}
\langle H_{k}y,y\rangle & \geq c\|y\|^{2}.\label{ya}
\end{align}
 
\end{lem}

\begin{proof}
The continuity of the Hessian of $F$ shows that there are small enough
numbers $\varepsilon>0$, and $\delta>0$ such that $\boldsymbol{H}_{F}-\varepsilon$
is positive in $B\left(z^{(k)},\delta\right)$ for every $k$ (and
we take $\delta$ small enough that these balls are disjoint). A first
order Taylor series expansion (with second order remainder in the
integral form) for the function $[0,1]\ni t\to F\left(tx+(1-t)z^{(k)}\right)$
shows that 
\begin{align}
F(x) & =\int_{0}^{1}\Big\langle(\mathbf{H}_{F})_{i,j}\left(tx+(1-t)z^{(k)}\right)\left(x-z^{(k)}\right),\left(x-z^{(k)}\right)\Big\rangle(1-t)\,\textup{d}t\geq\frac{1}{2}\left\langle \varepsilon\left(x-z^{(k)}\right),\left(x-z^{(k)}\right)\right\rangle ,\label{T61}
\end{align}
for every $x\in B\left(z^{(k)},\delta\right)$. For large enough $\|x\|$,
$F(x)$ is dominated by $\frac{1}{2}x^{T}J^{-1}x\geq c\|x\|^{2}$,
where $c>0$ is any constant smaller than the smallest eigenvalue
of $J^{-1}$. For large enough $x$, there is a constant $c$ such
that $\|x\|^{2}\geq c\|x-z^{(k)}\|^{2}$. We conclude that there is
a constant $c>0$ such that 
\begin{align}
F(x) & \geq c\left\Vert x-z^{(k)}\right\Vert ^{2}\label{T62}
\end{align}
for large enough $x$ and every $k$. Equations \eqref{T61} and \eqref{T62}
imply \eqref{T60} for every $x$ in the complement of a compact set
$\mathcal{C}$ not intersecting $B\left(z^{(k)},\delta\right)$, for
every $k$. On $\mathcal{C}$, $F-\alpha$ attains its minimum, which
we call $\beta>0$ (it cannot be smaller or equal 0 because $F$ attains
its minimum only at the points $z^{(k)}$). On the set $\mathcal{C}$,
\eqref{T60} holds true as well. This is verified with the following
estimation: 
\begin{align}
F(x)-\alpha & \geq\beta\geq\beta\frac{1}{\max_{x\in\mathcal{C}}(\|x-z_{k}\|^{2})}\left\Vert x-z^{(k)}\right\Vert ^{2}\label{T63}
\end{align}
for every $x\in\mathcal{C}$.

We restate Theorem \ref{thm:low_temp} for the reader's convenience:
\end{proof}
\begin{thm}
Assume the model is in the low temperature regime and non-critical
as per Definition \ref{def:non-crit}, i.e. the Hessian $H_{k}$ of
$F$ at $z_{k}$ is positive definite for all $k$. Then the function
$F$ defined in \eqref{eq:F} has a finite number of minima, $z^{(1)},\ldots,z^{(K)}\in\IR^{M}$.
There is a positive constant $\boldsymbol{C}_{\textup{low}}$ such
that for all $N\in\IN$
\[
\left|\mathbb{E}\left(X_{i}X_{j}\right)-\left\langle \vec{Z}_{\iota(i)},\vec{Z}_{\iota(j)}\right\rangle \right|\leq\boldsymbol{C}_{\textup{low}}\frac{\left(\ln N\right)^{\left(M+3\right)/2}}{\sqrt{N}},\quad i,j\in\IN_{N},i\neq j.
\]
\end{thm}

Recall the definitions of $Z_{2}(N)$, and $Z_{0}(N)$ in \eqref{Z2N}
and \eqref{Z0N}, as well as the expression 
\[
\mathbb{E}\left(X_{i}X_{j}\right)=\frac{Z_{2}(N)}{Z_{0}(N)}
\]

given in Theorem \ref{thm:de_Finetti_rep}.

We set (cf. Lemma \ref{LChin}) 
\begin{align*}
\rho_{N} & :=\sqrt{\frac{(M+3)\ln(N)}{c(N-1)}},
\end{align*}
which is selected in order to fulfil $e^{-(N-1)c\rho_{N}^{2}}\leq\frac{1}{N^{M+3}}$
(and we assume that $N$ is large enough such that $\rho_{N}<\delta$).
Lemma \ref{LChin} implies that
\begin{align}
\int_{\Big(\bigcup_{k=1}^{K}B(z_{k},\rho_{N})\Big)^{c}} & e^{-N(F(x)-\alpha)}\tanh^{2}\left(\frac{x_{\iota(i)}}{\sqrt{\alpha_{\iota(i)}}}\right)\,\textup{d}x\nonumber \\
\leq & \sum_{k=1}^{K}\int_{\Big(B(z_{k},\delta)\setminus B(z_{k},\rho_{N})\Big)\cup\Big(\bigcup_{j=1}^{K}B(z_{j},\delta)\Big)^{c}}e^{-N(F(x)-\alpha)}\tanh^{2}\left(\frac{x_{\iota(i)}}{\sqrt{\alpha_{\iota(i)}}}\right)\,\textup{d}x\nonumber \\
\leq & \sum_{k=1}^{K}e^{-\left(N-1\right)c\rho_{N}^{2}}\int_{\mathbb{R}^{M}}e^{-c\left\Vert x-z^{(k)}\right\Vert ^{2}}\tanh^{2}\left(\frac{x_{\iota(i)}}{\sqrt{\alpha_{\iota(i)}}}\right)\,\textup{d}x\nonumber \\
\leq & \frac{1}{N^{M+3}}\sum_{k=1}^{K}\int_{\mathbb{R}^{M}}e^{-c\left\Vert x-z^{(k)}\right\Vert ^{2}}\tanh^{2}\left(\frac{x_{\iota(i)}}{\sqrt{\alpha_{\iota(i)}}}\right)\,\textup{d}x\leq C\frac{1}{N^{M+3}},\label{to1p}
\end{align}
and likewise we conclude that 
\begin{align}
\int_{\Big(\bigcup_{k=1}^{K}B(z_{k},\rho_{N})\Big)^{c}}e^{-N(F(x)-\alpha)}\,\textup{d}x & \leq C\frac{1}{N^{M+3}}\label{to2p}
\end{align}
and (see \eqref{ya}) 
\begin{align}
\sum_{k=1}^{K}\int_{B(z_{k},\rho_{N})^{c}}e^{-\frac{N}{2}\langle H_{k}(x-z^{(k)}),(x-z^{(k)})\rangle}\left(\|x\|^{2}+1\right)\,\textup{d}x & \leq C\frac{1}{N^{M+3}}.\label{to2.1p}
\end{align}
Notice that $\frac{\textup{d}}{\textup{d}s}\ln\cosh(s)=\tanh(s)$
and $\frac{\textup{d}}{\textup{d}s}\tanh(s)=1-\tanh^{2}(s)$. This
implies that all even derivatives of $\ln\cosh(s)$ are polynomials
with even powers of $\tanh(s)$ and all odd derivatives of $\ln\cosh(s)$
are polynomials with odd powers of $\tanh(s)$ . Therefore, all such
derivatives are uniformly bounded in $\mathbb{R}$. Next, we recall
that $F\left(z^{(k)}\right)=\alpha$ and the Hessian of $F$ at $z^{(k)}$
is $H_{k}$. A second order Taylor series expansion (with third order
remainder in the mean value form) for the function $[0,1]\ni t\to F\left(tx+(1-t)z^{(k)}\right)$
shows that, for every $k$, 
\begin{align}
\left|F(x)-\alpha-\frac{1}{2}\left\langle H_{k}\left(x-z^{(k)}\right),x-z^{(k)}\right\rangle \right| & \leq C\left\Vert x-z^{(k)}\right\Vert ^{3},\label{to3p}
\end{align}
on $B\left(z^{(k)},\delta\right)$. Notice that for positive $s$
and $t$ with $s\leq t$ we have that the mean value theorem implies
that $\left|e^{-s}-e^{-t}\right|=\left|e^{-s}\right|\left|e^{s-t}-1\right|\leq e^{-s}|s-t|\leq|s-t|$.
This, together with the fact that $\tanh$ is bounded and \eqref{to3p},
implies that
\begin{align}
\left|\,\int_{B\left(z^{(k)},\rho_{N}\right)}e^{-N(F(x)-\alpha)}\tanh^{2}\left(\frac{x_{\iota(i)}}{\sqrt{\alpha_{\iota(i)}}}\right)\,\textup{d}x\right. & -\left.\int_{B(\rho_{N},z_{k})}e^{-\frac{N}{2}\langle H_{k}(x-z^{(k)}),(x-z^{(k)})\rangle}\tanh^{2}\left(\frac{x_{\iota(i)}}{\sqrt{\alpha_{\iota(i)}}}\right)\,\textup{d}x\,\right|\nonumber \\
 & \leq CN\int_{B\left(z^{(k)},\rho_{N}\right)}\left|x-z^{(k)}\right|^{3}\,\textup{d}x\leq CNr_{N}^{3+M}\nonumber \\
 & \leq CN\left(\ln N\right)^{(3+M)/2}\frac{1}{N^{(3+M)/2}}\leq C\left(\ln N\right)^{(3+M)/2}\frac{1}{N^{(1+M)/2}}\label{to2.2p}
\end{align}
Likewise, we get 
\begin{align}
\left|\,\int_{B\left(z^{(k)},\rho_{N}\right)}e^{-N(F(x)-\alpha)}\,\textup{d}x-\int_{B\left(z^{(k)},\rho_{N}\right)}e^{-\frac{N}{2}\langle H_{k}(x-z^{(k)}),(x-z^{(k)})\rangle}\,\textup{d}x\,\right| & \leq C\left(\ln N\right)^{(3+M)/2}\frac{1}{N^{(1+M)/2}}.\label{to2cp}
\end{align}
We denote by 
\begin{align*}
T_{2}\left(\tanh^{2},z_{\iota(i)}^{(k)}\right)(s) & :=\tanh^{2}\left(\frac{z_{\iota(i)}^{(k)}}{\sqrt{\alpha_{\iota(i)}}}\right)+\left(\tanh^{2}\right)'\left(\frac{z_{\iota(i)}^{(k)}}{\sqrt{\alpha_{\iota(i)}}}\right)\left(s-\frac{z_{\iota(i)}^{(k)}}{\sqrt{\alpha_{\iota(i)}}}\right)\\
 & \quad+\frac{1}{2}\left(\tanh^{2}\right)''\left(\frac{z_{\iota(i)}^{(k)}}{\sqrt{\alpha_{\iota(i)}}}\right)\left(s-\frac{z_{\iota(i)}^{(k)}}{\sqrt{\alpha_{\iota(i)}}}\right)^{2}
\end{align*}
the second order Taylor series of $\tanh$ at $z_{\iota(i)}^{(k)}$.
Using that all derivatives of $\tanh^{2}$ are uniformly bounded,
Taylor's Theorem implies that 
\begin{align*}
\left|\tanh(s)-T_{2}\left(\tanh^{2},z_{\iota(i)}^{(k)}\right)\left(x_{\iota(i)}\right)\right| & \leq C\left|x_{\iota(i)}-z_{\iota(i)}^{(k)}\right|^{3},
\end{align*}
and this implies that 
\begin{align}
 & \quad\left|\,\int_{B\left(z^{(k)},\rho_{N}\right)}e^{-\frac{N}{2}\langle H_{k}(x-z^{(k)}),(x-z^{(k)})\rangle}\tanh^{2}\left(\frac{x_{\iota(i)}}{\sqrt{\alpha_{\iota(i)}}}\right)\,\textup{d}x\right.\nonumber \\
 & \quad\left.-\int_{B\left(z^{(k)},\rho_{N}\right)}e^{-\frac{N}{2}\langle H_{k}(x-z^{(k)}),(x-z^{(k)})\rangle}T_{2}\left(\tanh^{2},z_{\iota(i)}^{(k)}\right)\left(\frac{x_{\iota(i)}}{\sqrt{\alpha_{\iota(i)}}}\right)\,\textup{d}x\,\right|\nonumber \\
 & \leq C\int_{B\left(z^{(k)},\rho_{N}\right)}\left\Vert x-z^{(k)}\right\Vert ^{3}\,\textup{d}x\leq Cr_{N}^{3+M}\nonumber \\
 & \leq C\left(\ln N\right)^{(3+M)/2}\frac{1}{N^{(3+M)/2}}.\label{to4p}
\end{align}

Next, since the function $y\mapsto e^{-N/2\langle H_{k}y,y\rangle}y_{\lambda}$
is odd and its integral vanishes, Proposition \ref{MGauss} implies
that
\begin{align}
\int_{\mathbb{R}^{M}} & e^{-\frac{N}{2}\langle H_{k}(x-z^{(k)}),(x-z^{(k)})\rangle}T_{2}\left(\tanh^{2},z_{\iota(i)}^{(k)}\right)\left(x_{\iota(i)}\right)\,\textup{d}x\nonumber \\
 & =\sum_{k=1}^{K}\left(\frac{(2\pi)^{M/2}}{\det(H_{k})^{1/2}}\frac{1}{N^{M/2}}\right)\Big(\tanh^{2}\left(\frac{z_{\iota(i)}^{(k)}}{\sqrt{\alpha_{\iota(i)}}}\right)+\frac{1}{N}\frac{1}{2}\left(\tanh^{2}\right)''\left(\frac{z_{\iota(i)}^{(k)}}{\sqrt{\alpha_{\iota(i)}}}\right)\left(H_{k}^{-1}\right)_{\iota(i),\iota(i)}\Big).\label{gausintp}
\end{align}
It follows from \eqref{Z2N}, \eqref{to1p}, \eqref{to2.1p}, \eqref{to2.2p},
\eqref{to4p}, and \eqref{gausintp} that
\begin{align}
 & \quad\left|Z_{2}(N)-\sum_{k=1}^{K}\Big(\frac{(2\pi)^{M/2}}{\det(H_{k})^{1/2}}\frac{1}{N^{M/2}}\Big)\Big(\tanh^{2}\left(\frac{z_{\iota(i)}^{(k)}}{\sqrt{\alpha_{\iota(i)}}}\right)+\frac{1}{N}\frac{1}{2}\left(\tanh^{2}\right)''\left(\frac{z_{\iota(i)}^{(k)}}{\sqrt{\alpha_{\iota(i)}}}\right)\left(H_{k}^{-1}\right)_{\iota(i),\iota(i)}\Big)\right|\nonumber \\
 & \leq C\left(\ln N\right)^{(3+M)/2}\frac{1}{N^{(1+M)/2}}\label{cas1p}
\end{align}
and it follows from \eqref{Z0N}, \eqref{to2p}, \eqref{to2.1p},
\eqref{to2cp}, and Proposition \ref{MGauss} that 
\begin{align}
\left|Z_{0}(N)-\sum_{k=1}^{K}\frac{(2\pi)^{M/2}}{\det(H_{k})^{1/2}}\frac{1}{N^{M/2}}\right| & \leq C\left(\ln N\right)^{(3+M)/2}\frac{1}{N^{(1+M)/2}}.\label{cas2p}
\end{align}
Eqs. \eqref{cas1p} and \eqref{cas2p} imply that 
\begin{align}
\left|\frac{Z_{2}(N)}{Z_{0}(N)}-\frac{\sum_{k=1}^{K}\frac{1}{\det(H_{k})^{1/2}}\Big(\tanh^{2}\left(\frac{z_{\iota(i)}^{(k)}}{\sqrt{\alpha_{\iota(i)}}}\right)+\frac{1}{N}\frac{1}{2}\left(\tanh^{2}\right)''\left(\frac{z_{\iota(i)}^{(k)}}{\sqrt{\alpha_{\iota(i)}}}\right)\left(H_{k}^{-1}\right)_{\iota(i),\iota(i)}\Big)}{\sum_{k=1}^{K}\frac{1}{\det(H_{k})^{1/2}}}\right| & \leq C\left(\ln N\right)^{(3+M)/2}\frac{1}{\sqrt{N}}.\label{cas3-1}
\end{align}
Eqs. \eqref{cas3-1} and \eqref{Deltalam} lead to the desired result.

\section*{Acknowledgements}

This research was supported by CONACYT, FORDECYT-PRONACES 429825/2020
(proyecto apoyado por el FORDECYT-PRONACES, FORDECYT-PRONACES 429825/2020),
recently renamed Project CF-2019/429825, and by the project PAPIIT-DGAPA-UNAM
IN101621. M.\! B.\! is a Fellow and G.\! T.\! is a Candidate of
the Sistema Nacional de Investigadores. G.\! T.\! was supported
by a Conahcyt postdoctoral fellowship.

\appendix

\section{Appendix}

Recall the definitions of $Z_{2}(N)$ and $Z_{0}(N)$ from \eqref{Z2N}
and \eqref{Z0N} for fixed $i,j\in\IN_{N}$ with $i\neq j$
\begin{align*}
Z_{2}(N) & :=Z_{2}(N;i,j):=\int_{\IR^{M}}\exp\left(-NF(x)\right)\tanh x_{\iota(i)}\tanh x_{\iota(j)}\textup{d}x
\end{align*}
and 
\begin{align*}
Z_{0}(N) & :=\int_{\IR^{M}}\exp\left(-NF(x)\right)\textup{d}x.
\end{align*}

The following theorem is a special case of \cite[Theorem 32]{KirsToth2022b}
a result proved for the first time in \cite{To2020phd}. We present
it here with a short proof for the reader's convenience.
\begin{thm}
\label{thm:de_Finetti_rep}The following identity holds
\begin{align*}
\IE\left(X_{i}X_{j}\right) & =\frac{Z_{2}(N)}{Z_{0}(N)},\quad i,j\in\IN_{N},i\neq j.
\end{align*}
\end{thm}

\begin{proof}
The Gaussian integral $1=\frac{1}{\sqrt{2\pi}}\int e^{-(s-a)^{2}/2}\,\textup{d}s$
implies, expanding the square, that 
\[
e^{a^{2}/2}=\frac{1}{\sqrt{2\pi}}\int e^{-s^{2}/2+sa}\,\textup{d}s.
\]
Similarly, we obtain for any positive definite matrix $A\in\IR^{M\times M}$
\begin{align}
e^{y^{T}Ay/2N} & =\frac{\sqrt{\det(A)}}{\left(2\pi N\right)^{M/2}}\int e^{-w^{T}Aw/(2N)+w^{T}Ay/N}\,\textup{d}w=\frac{N^{M/2}}{\left(2\pi\right)^{M/2}\sqrt{\det(A)}}\int e^{-(N/2)u^{T}A^{-1}u+u^{T}s({\bf x})}\,\textup{d}u,\label{Gauss}
\end{align}
the latter equality being the result of a change of variables $u=N^{-1}Aw$.
The first equality is straightforward for diagonal matrices $A$ (using
the one-dimensional version). Since $A$ is self-adjoint, we can diagonalize
it: $A=U^{T}DU$, for an orthogonal matrix $U$ and a diagonal matrix
$D$. A change of variables $z=Uw,$ gives the desired result (using
the diagonal case with $Uy$ instead of $y$).

For every 
\[
{\bf x}=\left(x_{1},\ldots,x_{N}\right)\in\{-1,1\}^{N},
\]
we set 
\[
s_{l}({\bf x}):=\frac{1}{\sqrt{\alpha_{l}}}\sum_{i\in\iota^{-1}\left(l\right)}x_{i},\quad\tilde{s}_{l}({\bf x}):=\frac{1}{\sqrt{N}}s_{l}({\bf x}),\quad s({\bf x}):=\left(s_{1}({\bf x}),\ldots,s_{M}({\bf x})\right)^{T},\quad\text{and }\tilde{s}({\bf x}):=\left(\tilde{s}_{1}({\bf x}),\ldots,\tilde{s}_{M}({\bf x})\right)^{T}.
\]
Recall Definition \ref{def:CWM} and note that
\[
e^{-\mathbb{H}({\bf x})}=e^{\tilde{s}({\bf x})^{T}J\tilde{s}({\bf x})/2}=e^{s({\bf x})^{T}Js({\bf x})/\left(2N\right)}.
\]
Using \eqref{Gauss} with $A=J$ and $y=s({\bf x})$ and the change
of variables $u=N^{-1}Jw$, we obtain 
\begin{align*}
e^{-\mathbb{H}({\bf x})} & =\frac{N^{M/2}}{\left(2\pi\right)^{M/2}\sqrt{\det(J)}}\int e^{-(N/2)u^{T}J^{-1}u+u^{T}s({\bf x})}\,\textup{d}u.
\end{align*}
Noting that the partition function $Z$ is given by $Z=\sum_{{\bf x}}e^{-\mathbb{H}({\bf x})}$,
we get 
\begin{align}
\mathbb{E}\left(X_{i}X_{j}\right) & =\frac{\int e^{-(N/2)u^{T}J^{-1}u}\sum_{{\bf x}}x_{i}x_{j}e^{u^{T}s({\bf x})}\,\textup{d}u}{\int e^{-(N/2)u^{T}J^{-1}u}\sum_{{\bf x}}e^{u^{T}s({\bf x})}\,\textup{d}u}.\label{E2}
\end{align}
We finish by proving that the numerator is $Z_{2}(N)$ and the denominator
$Z_{0}(N)$. The equation
\[
x_{i}x_{j}e^{u^{T}s({\bf x})}=x_{i}e^{u_{\iota(i)}x_{i}/\sqrt{\alpha_{\iota(i)}}}x_{j}e^{u_{\iota(j)}x_{j}/\sqrt{\alpha_{\iota(j)}}}\prod_{k\notin\{i,j\}}e^{u_{\iota(k)}x_{k}/\sqrt{\alpha_{\iota(k)}}}
\]
and the associativity and commutativity of addition yield
\begin{align*}
\sum_{{\bf x}}x_{i}x_{j}e^{u^{T}s({\bf x})} & =\sum_{x_{1}\in\{-1,1\}}\cdots\sum_{x_{N}\in\{-1,1\}}x_{i}x_{j}e^{u^{T}s({\bf x})}\\
 & =\sum_{x_{i}\in\{-1,1\}}x_{i}e^{u_{\iota(i)}x_{i}/\sqrt{\alpha_{\iota(i)}}}\sum_{x_{j}\in\{-1,1\}}x_{j}e^{u_{\iota(j)}x_{j}/\sqrt{\alpha_{\iota(j)}}}\prod_{k\notin\{i,j\}}\sum_{k\in\{-1,1\}}e^{u_{\iota(k)}x_{k}/\sqrt{\alpha_{\iota(k)}}}.
\end{align*}
Then we note that
\begin{align*}
\sum_{x_{i}\in\{-1,1\}}x_{i}e^{u_{\iota(i)}x_{i}/\sqrt{\alpha_{\iota(i)}}} & =e^{u_{\iota(i)}/\sqrt{\alpha_{\iota(i)}}}-e^{-u_{\iota(i)}/\sqrt{\alpha_{\iota(i)}}}\\
 & =2\sinh\left(\frac{u_{\iota(i)}}{\sqrt{\alpha_{\iota(i)}}}\right)=2\tanh\left(\frac{u_{\iota(i)}}{\sqrt{\alpha_{\iota(i)}}}\right)\cosh\left(\frac{u_{\iota(i)}}{\sqrt{\alpha_{\iota(i)}}}\right),
\end{align*}
since $\tanh t=\sinh t/\cosh t$ for any $t\in\IR$. Similarly, 
\[
\sum_{x_{j}\in\{-1,1\}}x_{j}e^{u_{\iota(j)}x_{j}}=2\tanh\left(\frac{u_{\iota(j)}}{\sqrt{\alpha_{\iota(j)}}}\right)\cosh\left(\frac{u_{\iota(j)}}{\sqrt{\alpha_{\iota(j)}}}\right),
\]
and for all $k\notin\left\{ i,j\right\} $ 
\[
\sum_{k\in\{-1,1\}}e^{u_{\iota(k)}x_{k}/\sqrt{\alpha_{\iota(k)}}}=e^{u_{\iota(k)}/\sqrt{\alpha_{\iota(k)}}}+e^{-u_{\iota(k)}/\sqrt{\alpha_{\iota(k)}}}=2\cosh\left(\frac{u_{\iota(k)}}{\sqrt{\alpha_{\iota(k)}}}\right).
\]
We have
\begin{align*}
 & \quad\int e^{-(N/2)u^{T}J^{-1}u}\sum_{{\bf x}}x_{i}x_{j}e^{u^{T}s({\bf x})}\,\textup{d}u\\
 & =2^{N}\int e^{-(N/2)u^{T}J^{-1}u}\tanh\left(\frac{u_{\iota(i)}}{\sqrt{\alpha_{\iota(i)}}}\right)\tanh\left(\frac{u_{\iota(j)}}{\sqrt{\alpha_{\iota(j)}}}\right)\prod_{l=1}^{M}\cosh^{N_{l}}\left(\frac{u_{l}}{\sqrt{\alpha_{l}}}\right)\,\textup{d}u\\
 & =2^{N}\int e^{-N\left(\frac{1}{2}u^{T}J^{-1}u-\sum_{l=1}^{M}\alpha_{l}\ln\cosh\left(\frac{u_{l}}{\sqrt{\alpha_{l}}}\right)\right)}\tanh\left(\frac{u_{\iota(i)}}{\sqrt{\alpha_{\iota(i)}}}\right)\tanh\left(\frac{u_{\iota(j)}}{\sqrt{\alpha_{\iota(j)}}}\right)\,\textup{d}u,
\end{align*}
where we again used that $N_{l}=\alpha_{l}N$ for all $l\in\IN_{M}$.
Similarly, we obtain an expression for the denominator of \eqref{E2}:
\[
\int e^{-(N/2)u^{T}J^{-1}u}\sum_{{\bf x}}e^{u^{T}s({\bf x})}\,\textup{d}u=2^{N}\int e^{-N\left(\frac{1}{2}u^{T}J^{-1}u-\sum_{l=1}^{M}\alpha_{l}\ln\cosh\left(\frac{u_{l}}{\sqrt{\alpha_{l}}}\right)\right)}\,\textup{d}u.
\]
The result follows by dividing numerator and denominator by $2^{N}$.
\end{proof}
The following result pertaining to Gaussian integrals is used repeatedly
throughout the proofs. It is an immediate consequence of the definition
of a multivariate normal distribution. See, e.g., \cite[pp. 176-177]{Dur2019}
for a reference.
\begin{prop}
\label{MGauss}Fix $d\in\IN$. We have for all invertible $A\in\IR^{d\times d}$

\begin{align*}
\frac{1}{(2\pi)^{d/2}}\sqrt{\det(A)}\int e^{-\frac{1}{2}z^{T}Az}z^{T}z\,\textup{d}z=A^{-1}, & \hspace{1cm}\frac{1}{(2\pi)^{d/2}}\sqrt{\det(A)}\int e^{-\frac{1}{2}z^{T}Az}\,\textup{d}z=1.
\end{align*}
\end{prop}

\bibliographystyle{plain}
\bibliography{../References}

\end{document}